\documentclass[12pt]{amsart}

\usepackage{color}
\usepackage{xcolor}
\usepackage{amsxtra}
\usepackage{mathrsfs}
\usepackage{yfonts}
\usepackage{amsmath}
\usepackage{amsthm}
\usepackage{amsfonts}
\usepackage{amssymb}
\usepackage{amscd,gensymb}
\usepackage{graphics}
\usepackage{graphicx}
\usepackage{xcolor}
\usepackage{relsize}

\usepackage[bottom=1in,top=0.5in]{geometry}

\numberwithin{equation}{section}

\newtheorem{theorem}{Theorem}[section]
\newtheorem{lem}[theorem]{Lemma}
\newtheorem{prop}[theorem]{Proposition}
\newtheorem{cor}[theorem]{Corollary}
\newtheorem{deff}[theorem]{Definition}
\newtheorem{thm}[theorem]{Theorem}
\newtheorem*{theorem*}{Theorem}
\newtheorem*{theorem**}{Claim}
\newtheorem*{theorem***}{}

\theoremstyle{remark}
\newtheorem{remark}[theorem]{Remark}

\newcommand{\D}{\mathscr{D}^s_{\mu}}

\newcommand{\Diff}{\mathscr{D}}
\newcommand{\He}{\mathscr{A}_{\mu}^s}
\newcommand{\Hel}{\mathscr{A}_{\mu, 0}^s}
\newcommand{\reals}{\mathbb{R}}

\newcommand{\ad}{\mathrm{ad}}
\newcommand{\grad}{\nabla}

\newcommand{\Ad}{\mathrm{Ad}}
\newcommand{\Tr}{\mathrm{Tr}}

\DeclareMathOperator{\curl}{curl}
\DeclareMathOperator{\diver}{div}

\newcommand\restr[2]{{
		\left.\kern-\nulldelimiterspace 
		#1 
		\vphantom{\big|} 
		\right|_{#2} 
}}

\begin{document}
	
	\title[Axisymmetric diffeomorphisms and ideal fluids]{Axisymmetric diffeomorphisms and ideal fluids
	\\ on Riemannian 3-manifolds}
	\author{Leandro Lichtenfelz} 
	\address{Department of Mathematics, University of Pennsylvania, Philadelphia, PA 19104, USA} 
	\email{llichte2@gmail.com} 
	\author{Gerard Misio{\l}ek} 
	\address{Department of Mathematics, University of Notre Dame, Notre Dame, IN 46556, USA} 
	\email{gmisiole@nd.edu} 
	\author{Stephen C. Preston} 
	\address{Department of Mathematics, Brooklyn College, CUNY, New York, NY 10468 
	and Graduate Center, CUNY, New York, NY 10016, USA} 
	\email{Stephen.Preston@brooklyn.cuny.edu} 


	\date{\today}
	
	\maketitle

	\begin{abstract} 
	We study the Riemannian geometry of 3D axisymmetric ideal fluids. 
	We prove that the $L^2$ exponential map on the group of volume-preserving diffeomorphisms 
	of a $3$-manifold is Fredholm along axisymmetric flows with sufficiently small swirl. 
	Along the way, we define the notions of axisymmetric and swirl-free diffeomorphisms of any manifold 
	with suitable symmetries and show that such diffeomorphisms form a totally geodesic submanifold 
	of infinite $L^2$ diameter inside the space of volume-preserving diffeomorphisms whose diameter 
	is known to be finite. 
	As examples we derive the axisymmetric Euler equations on $3$-manifolds equipped with 
	each of Thurston's eight model geometries. 
	
	\end{abstract}
	
	\tableofcontents

	\section{ Introduction }
	\label{sec:Introduction}
	%
	%
	
	Let $M$ be an oriented Riemannian manifold (possibly with boundary) of dimension $n$
	and let $\mu$ be its Riemannian volume.
	Let $\mathscr{D}_\mu$ be the group of smooth diffeomorphisms of $M$ preserving the volume form.
	Consider the right-invariant (weak) Riemannian metric on $\mathscr{D}_\mu$ which
	at the identity diffeomorphism $e$ is given by the $L^2$ inner product
	\begin{equation} \label{eq:L2}
	\langle v, w \rangle_{L^2}
	=
	\int_M \langle v(x), w(x) \rangle \, d\mu,
	\qquad\quad
	v, w \in T_e\mathscr{D}_\mu.
	\end{equation}

	In a celebrated paper V. Arnold \cite{Ar1} showed that a geodesic $\gamma(t)$ of the $L^2$ metric
	in $\mathscr{D}_\mu$ starting at $\gamma(0)=e$ in the direction $u_0 \in T_e\mathscr{D }_\mu$
	corresponds to the solution of the Cauchy problem for the incompressible Euler equations in $M$,
	namely
	\begin{align} \label{eq:Euler} \nonumber
	&\partial_t u + \nabla_u u = - \nabla p
	\\
	&\mathrm{div}\, u = 0
	\\ \nonumber
	&u(0) = u_0
	\end{align}
	where $p$ is the pressure function and $u$ is the fluid velocity (parallel to the boundary
	if $\partial M \neq 0$).
	The correspondence is obtained by solving the initial value problem for the flow equation
	\begin{align} \label{eq:Flow}
	&\frac{d\gamma}{dt}(t,x) = u(t,\gamma(t,x))
	\\ \nonumber
	&\gamma(0,x)=x,
	\qquad\qquad\quad
	x\in M.
	\end{align}

	The two formulations of the equations of fluid motion are only formally equivalent.
	It turns out that if $\mathscr{D}_\mu$ is equipped with a suitable topology then the geodesic formulation
	has an important technical advantage.
	For example, Ebin and Marsden \cite{EbinMarsden} showed that the completion of
	the group of diffeomorphisms in the Sobolev $H^s$ topology with $s>n/2+1$
	becomes a smooth Hilbert manifold $\mathscr{D}^s_\mu$ and a topological group.
	In this case the geodesic equation can be uniquely solved (at least for short times)
	by the method of successive approximations for ODEs in Banach spaces.
	Consequently, the geodesics of the $L^2$ metric depend smoothly (in the $H^s$ topology)
	on the initial data and there is a well-defined smooth exponential map
	\begin{equation} \label{eq:exp}
	\exp_e : \mathcal{U} \subseteq T_e\mathscr{D}^s_\mu \to \mathscr{D}^s_\mu,
	\qquad
	\exp_e tu_0 := \gamma(t)
	\end{equation}
	where $\gamma(t)$ is the unique geodesic of the metric \eqref{eq:L2}
	with $\gamma(0)=e$ and $\dot{\gamma}(0)=u_0$ defined on some open subset $\mathcal{U}$.
	A similar calculation as in finite dimensions shows that ${d}\exp_e(0)=\mathrm{id}$
	and therefore the $L^2$ exponential map is a local diffeomorphism near the identity in $\mathscr{D}^s_\mu$
	by the inverse function theorem.
	Furthermore, if the underlying manifold $M$ is two-dimensional then this map can be extended to
	the whole tangent space $T_e\mathscr{D}^s_\mu$.\footnote{This is essentially a consequence of
	a theorem of Wolibner \cite{wol}.}
	
	The study of the structure of singularities of the map \eqref{eq:exp} is of considerable interest.
	The question whether geodesics in $\mathscr{D}^s_\mu$ have conjugate points was raised by
	Arnold \cite{Ar1, Ar2, AK} and, subsequently, various examples have been constructed by
	Misio{\l}ek \cite{Mi92, Mi96}, Shnirelman \cite{Shn94}, Preston \cite{Pr06}, Benn \cite{Benn1, Benn2}
	and others.
	Further progress was made by Ebin et al. \cite{EMP06} who proved that if $n=2$ then
	the $L^2$ exponential map is necessarily a non-linear Fredholm map of index zero. In this case,
	it is possible to obtain normal forms for the $L^2$ exponential map near its singularities (see Lichtenfelz \cite{L}).
	Moreover, when $M$ is the flat 2-torus Shnirelman \cite{Shn05} showed that it is a quasiruled Fredholm map.
	However, the picture changes in dimension $n=3$. Simple examples of a steady rotation of
	the solid torus in $\mathbb{R}^3$ or a left-invariant Killing field on the 3-sphere show that the Fredholm property
	may fail in general, cf. \cite{EMP06}.
	Further examples can be found in \cite{Pr06}
	and, more recently, in Preston and Washabaugh \cite{PrWash}.
	Yet despite these examples the failure of Fredholmness seems to be borderline and there is some evidence
	that exponential maps of right-invariant $H^r$ metrics of positive Sobolev index $r>0$ are nonlinear Fredholm
	of index zero, cf. \cite{MP10}.
	
	In this paper we study the $L^2$ exponential map in the three-dimensional case and prove that
	it retains the Fredholm property when restricted to certain subsets of axisymmetric diffeomorphisms of 3-manifolds.
	The proof of Fredholmness when $n=2$ relies on compactness properties of the algebra coadjoint operator.
	The fact that the geometry of coadjoint orbits of volume-preserving diffeomorphisms of a 3-manifold
	is considerably more complicated (a source of many difficulties encountered in 3D hydrodynamics)
	is one of the reasons why the structure of singularities of the $L^2$ exponential map is not yet
	completely understood when $n=3$.
	On the other hand, if the manifold admits symmetries then the situation simplifies.
	For example, if the velocity field initially commutes with an isometry of the underlying manifold
	then it does so necessarily for all time. When combined with the fact that the swirl is also conserved in time
	this essentially reduces the problem from three to two spatial dimensions. A perturbation argument then
	yields Fredholmness of the exponential map for axisymmetric ideal fluids with suitably small swirls.
	In Euclidean space there are two simple types of isometries:
	translations which yield equations that completely decouple
	and
	rotations which yield equations with a singularity located on the axis of rotation.
	In other geometries we find equations with nontrivial twisting but without singular behavior.
	A number of examples together with explicit calculations are included in the later sections of the paper.
	\smallskip
	
		The remainder of the paper has the following structure.
		After reviewing some necessary background in Section \ref{sec:FJ}, in Section \ref{sec:ax_diff}
		we develop a general framework for axisymmetric flows on Riemannian manifolds
		and show that the notion of swirl generalizes in a natural way.
		
		In Section \ref{sec:swirlfree} we construct the group of swirl-free diffeomorphisms, and show that it has infinite diameter, despite being a totally geodesic submanifold of $\D$ (which has finite diameter; see \cite{ShnVol}).
		
		In Section $\ref{section_fred}$ we prove that swirl is transported by the fluid and use this fact to extend
		the classical global wellposedness result for swirl-free flows in $\mathbb{R}^3$ to this setting.
		Finally, in Section $\ref{sec:ex}$ we present a number of examples of axisymmetric flows
		on 3-manifolds equipped with the model geometries of Thurston. One family of examples includes Boothby-Wang
fibrations over constant-curvature surfaces, for which we will show the Euler-Arnold equation is
\begin{equation}\label{fibration}
\partial_t \Delta f + \{f,\Delta f\} + \{f,\sigma\} = 0, \qquad \partial_t \sigma + \{f,\sigma\} = 0,
\end{equation}
where the Laplacian and Poisson bracket are computed on either the plane, the sphere, or the hyperbolic plane.

\textit{Acknowledgements.} We would like to thank David Ebin, Herman Gluck and Karsten Grove for helpful suggestions.
The work of S.C.P. was partially supported by Simons Foundation Collaboration Grant no. 318969.	
	
\section{Background and the setup}
	\label{sec:FJ}
	
	In this section we recall a few basic facts about Fredholm maps and describe the geometric setup
	used in the paper. Our main references regarding the geometry of the group of diffeomorphisms
	are \cite{EbinMarsden}, \cite{EMP06}, \cite{Mi92} and \cite{MP10}.
	
	A bounded linear operator between Banach spaces $X$ and $Y$ is said to be Fredholm if it has
	finite dimensional kernel and cokernel. Consequently, its range is closed by the open mapping theorem.
	If $X=Y$ then the set of Fredholm operators forms an open subset in the space of
	all bounded linear operators on $X$ which is invariant under products and adjoints.
	Furthermore, it is also stable under perturbations by compact operators as well as
	perturbations by bounded operators of suitably small norm.
	The index of an operator $L$ given by
	$
	\mathrm{ind}\, L = \dim\ker L - \dim\mathrm{coker}\, L
	$
	defines a map into $\mathbb{Z}$ which is constant on connected components of
	the set of Fredholm operators.\footnote{A good reference for these facts is Kato \cite{Kato}.}
	A nonlinear generalization of this notion was introduced by Smale \cite{Sm65}.
	A $C^1$ map $f$ between Banach manifolds is called Fredholm if its Fr\'echet derivative is a Fredholm operator
	at each point in the domain of $f$. If the domain is connected then the index of $f$ is by definition
	the index of its derivative.
	
	As in classical Riemannian geometry a singular value of the exponential map of an infinite-dimensional
	(weak) Riemannian manifold is called a conjugate point. However, in contrast with the finite dimensional case
	it is necessary to distinguish two types of such points depending on how the derivative of the exponential map
	fails to be an isomorphism.
	If $\gamma(t)$ is a geodesic then a point $q=\gamma(1)$ is called monoconjugate
	if $d \exp_p(t\dot{\gamma}(0))$
	is not injective as a linear map between the tangent spaces at $p$ and $q$;
	it is called epiconjugate if $d \exp_p(t\dot{\gamma}(0))$ is not surjective.
	Furthermore, it is not difficult to see that the order of conjugacy can be infinite and that finite geodesic segments
	may contain clusters of conjugate points of either type. A simple example of the former is provided by any pair of
	antipodal points on the unit sphere in a Hilbert space equipped with the round metric.
	An example of the latter can be found on the infinite-dimensional ellipsoid constructed by
	Grossman \cite{Grossman65}.
	Such phenomena are ruled out if the exponential map is Fredholm.
	
	Given $u_0 \in T_{e}\mathscr{D}_{\mu}^{s}$ let $\gamma(t)$ be the associated $L^2$ geodesic
	starting at the identity element in $\mathscr{D}^s_\mu$.
	A very useful tool in the study of the exponential map is the Jacobi equation along $\gamma(t)$,
	namely
	\begin{align*} 
	J'' + \mathcal{R}(J, \dot{\gamma})\dot{\gamma} = 0
	\end{align*}
	where $'$ denotes the covariant derivative of the $L^2$-metric in the direction of $\gamma$
	and $\mathcal{R}$ is the $L^2$ curvature operator along $\gamma$.
	It is known that $\mathcal{R}$ is a bounded tri-linear operator on each tangent space
	$T_\gamma\mathscr{D}^s_\mu$ from which it follows that the Cauchy problem for the Jacobi equation
	admits unique solutions which persist as long as the geodesic exists.
	In fact, solutions subject to the initial conditions $J(0)=0$ and $J'(0)=w$ provide precise information
	about the derivative of the exponential map by means of the formula
	$d \exp_e(tu_0) tw = J(t)$.
	In what follows it will be convenient to introduce the following operator on $T_e\mathscr{D}_\mu^s$
	\begin{align} \label{eq:Phi}
	w \mapsto \Phi_t(w) = dL_{\gamma(t)^{-1}} d\exp_e(tu_0) tw
	\end{align}
	where $L_\eta$ is the left multiplication in $\mathscr{D}^s_\mu$ by $\eta$ and decompose $\Phi_t$
	into a sum of two bounded linear operators as described in the following proposition.
	\begin{prop} \label{prop:OLD}
		$\Phi_t$ is a bounded linear operator on $T_{e}\mathscr{D}_{\mu}^{s}$ for each $t \geq 0$.
		Furthermore, if $u_0 \in T_e\mathscr{D}_\mu^{s+1}$ then 
		\begin{equation} \label{eq:solop}
		\Phi_t = \Omega_t - \Gamma_t
		\end{equation}
		where $\Omega_t$ and $\Gamma_t$ are bounded operators on $T_e\mathscr{D}_\mu^s$ given by
		\begin{align}
		&w \to \Omega_t w
		=
		\int_{0}^{t}\mathrm{Ad}_{\gamma_\tau^{-1}}\mathrm{Ad}_{\gamma_\tau^{-1}}^{*} w \, d\tau
		\label{eq:Omega} \\  \label{eq:Gamma}
		&w \to \Gamma_tw
		=
		\int_{0}^{t}\mathrm{Ad}_{\gamma_\tau^{-1}}
		\mathrm{Ad}_{\gamma_\tau^{-1}}^\ast \mathrm{ad}_{\Phi_\tau w}^\ast u_0 \, d\tau
		\end{align}
		and the two coadjoint operators are defined by the formulas
		$\langle \mathrm{Ad}_{\eta}^{*}v,w\rangle _{L^{2}} = \langle v,\mathrm{Ad}_{\eta}w \rangle _{L^{2}}$
		and
		$\langle \mathrm{ad}_{v}^{*}u,w \rangle _{L^{2}} = \langle u,\mathrm{ad}_{v}w \rangle _{L^2}$
		for any $u, v, w \in T_e\mathscr{D}_\mu^s$.
	\end{prop}
	\begin{proof}
		See \cite{MP10}; Thm. 5.6.
	\end{proof}
	%
	%
	%
	\begin{remark} \label{rem:modulo}
		We emphasize that on $\mathscr{D}^s_\mu$ the decomposition \eqref{eq:solop}-\eqref{eq:Gamma}
		must be applied with care due to the loss of derivatives involved in calculating the differential of
		the left multiplication operator $L_\eta$.
		Recall that if $s>n/2+1$ then left translations are continuous but not differentiable in the $H^s$ topology.
		For this reason in several places in the paper it will be convenient to assume that
		the initial velocity $u_0$ is an element of $T_e\mathscr{D}_\mu$
		and, in particular, the proof of Fredholmness of the exponential map \eqref{eq:exp}
		will be carried out under the assumption that the corresponding $L^2$ geodesic $\gamma(t)$
		is $C^\infty$ smooth, (cf. e.g., \cite{EbinMarsden}; Thm. 12.1).
		The proof of Fredholmness in the general case when $u_0 \in T_e\mathscr{D}^s_\mu$
		follows then by a simple perturbation argument which is completely analogous
		to that in \cite{EMP06} or \cite{MP10} and can therefore be omitted.
	\end{remark}
	%
	
\section{Axisymmetric diffeomorphisms of 3-manifolds}
	\label{sec:ax_diff}
	
	In this section we introduce the notion of axisymmetric diffeomorphisms of a general Riemannian manifold $M$
	of dimension $n=3$ equipped with a smooth Killing vector field $K$.
	To simplify the exposition we assume that $M$ is compact (possibly with boundary)
	or $M$ is the three-dimensional Euclidean space.
	The flow of $K$ consists of isometries and its covariant derivative is antisymmetric, i.e.,
	\begin{equation*} \label{killingdef}
	\langle \nabla_w K , v \rangle + \langle w, \nabla_v K \rangle = 0
	\end{equation*}
	for all vector fields $v$ and $w$ on $M$. Among the standard examples are the isometric $\mathbb{R}$-action
	and the circle action on the Euclidean space $\mathbb{R}^3$ with cylindrical coordinates $(r, \theta, z)$
	where $K = \partial_z$ and $K = \partial_{\theta}$, respectively.\footnote{To obtain a compact manifold
	for these examples, we may work on $M=D^2\times \mathbb{S}^1$, the vertically periodic cylinder of radius $1$
	and height $2\pi$.}
	
	A divergence-free vector field $u$ on $M$ will be called \emph{axisymmetric} if $[K, u] = 0$.
	The set of axisymmetric vector fields will be denoted by $T_e\He$.
	
	If $M=\mathbb{R}^3$ and $K = \partial_{\theta}$ then the commutator condition is precisely
	the statement that components of $u$ do not depend on the $\theta$-variable
	which corresponds to the standard case of rotational symmetry.
	If $M=\mathbb{R}^3$ and $K = \partial_z$ then the components of $u$ are independent of the $z$-variable
	which reflects translational symmetry. All other infinitesimal isometries of Euclidean space are linear combinations
    of these with more complicated equations.
	
	In non-Euclidean geometries axisymmetric flows are more interesting.
	In fact, the translational case of $\mathbb{R}^3$ is essentially trivial because in this case
	the Euler equations \eqref{eq:Euler} completely decouple,
	while the rotational case is only slightly more complicated since $K$ vanishes along the $z$-axis.
	On the other hand, on a curved manifold such as the $3$-sphere $\mathbb{S}^3$
	there exist nonvanishing Killing fields for which the Euler equations do not decouple.
	\begin{prop} \label{prop_lie_algebra}
		The space $T_e\mathscr{A}_\mu$ of smooth axisymmetric vector fields on $M$
		is an infinite-dimensional Lie algebra.
	\end{prop}
	\begin{proof}
		Since $T_{e}\Diff_{\mu}$ is an (infinite-dimensional) Lie algebra we need only check that $[u, v]$
		commutes with $K$ whenever $u, v \in T_e\mathscr{A}_\mu$. This however follows at once
		from the Jacobi identity
		$
		\big[ K, [u,v] \big] = - \big[ u,[v,K] \big] - \big[ v,[K,u] \big] = 0.
		$
	\end{proof}
	A volume-preserving diffeomorphism of $M$ will be called \emph{axisymmetric}
	if it commutes with the flow of the Killing field $K$.
	It turns out that the Sobolev $H^s$ completion of the set of all such diffeomorphisms
	$$
	\He = \big\{ \eta \in \D : \varphi_t \circ \eta = \eta \circ \varphi_t, \; \varphi_t \; \text{is a flow of $K$} \big\}
	$$
	can be equipped with the structure of a smooth Hilbert manifold whose tangent space at $\eta \in \He$ is
	\begin{gather}
	\begin{split} \label{eq_axi_tangent}
	T_\eta\He
	=
	d_eR_{\eta} ( T_e\He )
	=
	\big\{ v \circ \eta : v \in T_e\He \big\}.
	\end{split}\end{gather}

	If $M$ is compact the following proposition can be essentially deduced from the results of
	 Omori\footnote{In Omori's terminology $\mathscr{A}^\infty(M)$ is a strong ILH Lie group and a Sobolev manifold.
	 In the boundary case one needs to assume, for technical reasons, that $M$ is diffeomorphic to
	 $\partial M \times [0, 1)$ in some neighborhood of $\partial M$.}
	 \cite{O}; Chap. VIII, XVI.
	 The direct proof given below (following the unpublished paper \cite{EbinPreston}) has an advantage of being
	 easily adapted to noncompact manifolds (with asymptotically Euclidean ends).
	\begin{prop}\label{prop_axi_manifold}
		The set $\He$ is a topological group and a smooth Hilbert submanifold of $\D$.
	\end{prop}
	\begin{proof}
	        For simplicity we will assume that $\mu(M)=1$ and $\partial M = \emptyset$.
		We will first show that the group of all $H^s$ diffeomorphisms commuting with the flow $\varphi_t$,
		which we denote by $\mathscr{A}^s$, is a smooth submanifold of the group $\mathscr{D}^s$
		of all $H^s$ diffeomorphisms of $M$.
		The proof follows a well-known construction of Eells \cite{Ee}. We will show how to define a typical
		coordinate chart for $\mathscr{A}^s$.
		
		Let $\exp^{_M}$ be the Riemannian exponential map on $M$.
		For any $\eta \in \mathscr{A}^s$ the map $V \to \exp^{_M}_\eta \circ V$
		defines a coordinate chart around $\eta$ in $\mathscr{D}^s$ when restricted to
		a sufficiently small\footnote{E.g., assume that the $H^s$ norms of vectors in $\mathcal{U}$
		are less than the injectivity radius of $M$.}
		$H^s$ neighbourhood $\mathcal{U}$ of the zero section of the pull-back bundle $\eta^\ast TM$.
		Let
		$\mathfrak{S}_{\eta} = \{ v \circ \eta : [v, K] = 0, v \in H^s(TM) \}$.
		Clearly, $\mathfrak{S}_{\eta}$ is a closed subspace of $H^s(\eta^*(TM))$. Furthermore, we have
		\begin{theorem**} \;\;
			$
			\exp^{_M}_{\eta} (\mathcal{U} \cap \mathfrak{S}_{\eta} )
			=
			\exp^{_M}_{\eta} \mathcal{U} \cap \mathscr{A}^s.
			$
		\end{theorem**}
		To prove the claim, given $\xi \in \exp^{_M}_{\eta} (\mathcal{U} \cap \mathfrak{S}_{\eta} )$
		let $w$ be an $H^s$ vector field on $M$
		with $[w, K] = 0$ such that $\xi = \exp^{_M}_\eta (w\circ\eta)$.
		Since $w$ is axisymmetric, we have
	        $w\circ\varphi_t = \varphi_{t\ast e} w$
		and using the fact that $\varphi_t$ is an isometry commuting with $\eta$, we find
		\begin{align*}
		\xi \circ \varphi_t
		&= \exp^{_M}_{\eta \circ \varphi_t}(w\circ\eta \circ \varphi_t)
		=
		\exp^{_M}_{\varphi_t \circ \eta}(w\circ\varphi_t \circ \eta)
		\\
		&= \exp^{_M}_{\varphi_t \circ \eta}(\varphi_{t\ast\eta} (w\circ\eta))
		=
		\varphi_t \circ \exp^{_M}_\eta(w\circ\eta)
		\\
		&=
		\varphi_t \circ \xi.
		\end{align*}
		This shows the inclusion
		$\exp^{_M}_\eta(\mathcal{U}\cap\mathfrak{S}_\eta)
		\subseteq
		\exp^{_M}_\eta\mathcal{U} \cap \mathscr{A}^s$.
		
		On the other hand, given any $\xi \in \exp^{_M}_{\eta} \mathcal{U} \cap \mathscr{A}^s$
		pick a vector field $w \in H^s(TM)$ such that $\xi = \exp^{_M}_\eta(w\circ\eta)$ and
		$\xi \circ \varphi_t = \varphi_t \circ \xi$.
		To show that $w$ is axisymmetric we fix $t$ and consider two geodesics
		$$
		t' \to \gamma_1(t')
		=
		\exp^{_M}_{\eta \circ \varphi_t}(t' w\circ\eta \circ \varphi_t)
		=
		\exp^{_M}_\eta(t'w\circ\eta) \circ \varphi_t
		$$
		and
		$$
		t' \to \gamma_2(t')
		=
		\exp^{_M}_{\varphi_t \circ \eta}(t'\varphi_{t\ast\eta} (w\circ\eta))
		=
		\varphi_t \circ \exp^{_M}_\eta(t'w\circ\eta).
		$$
		Clearly, $\gamma_1(0) = \gamma_2(0)$ and since $\xi$ commutes with $\varphi_t$
		we also have $\gamma_1(1) = \gamma_2(1)$.
		Since by construction we can always arrange that $\|w\|_{H^s}$ is smaller than the injectivity radius of $M$
		(observe that $w$ is in $\mathcal{U}$)
		there is only one geodesic between $\gamma_1(0)$ and $\gamma_1(1)$ which implies that
		$\gamma_1 = \gamma_2$.
		Since both geodesics have the same initial velocity, we have
		$$
		w\circ\varphi_t \circ \eta = w\circ\eta \circ \varphi_t = \varphi_{t\ast\eta}( w\circ\eta)
		$$
		from which we deduce that $w\circ\varphi_t = \varphi_{t\ast e}w$ for all $t$ so that $w$ is axisymmetric.
		This shows that
		$\exp^{_M}_\eta\mathcal{U} \cap \mathscr{A}^s
		\subseteq
		\exp^{_M}_\eta(\mathcal{U}\cap\mathfrak{S}_\eta)$
		and establishes the claim.
		
		Next, let $H_{K, 1}^{s-1}(\Lambda^n M)$ be the set of all $n$-forms of Sobolev class $H^{s-1}$
		which are invariant under $\varphi_t$ and have volume $1$
		or, equivalently,
		let
		\begin{equation*}
		H_{K,\,1}^{s-1}(M) = \Big\{ f \in H^{s-1}(M): f \circ \varphi_t = f \;\; \text{and} \;\; \int_M f \, d\mu(x) = 1 \Big\}
		\end{equation*}
		be the set of $K$-invariant $H^{s-1}$ functions of mean $1$.
		Clearly, this set is a hyperplane of the space of all $H^{s-1}$ functions invariant under $K$
		which, in turn, form a closed subspace of $H^{s-1}(M)$.
		
		Consider the map
		$F : \mathscr{A}^s \to H_{K, 1}^{s-1}(M)$
		given by the Radon-Nikodym derivative
		$\eta \to F(\eta) = d(\eta^\ast\mu)/d\mu$.
		$F$ is a smooth map between Hilbert manifolds and we can set
		$\mathscr{A}_\mu^s = F^{-1}(1)$.
		The derivative of $F$ at $\eta \in \mathscr{A}^s$ is
		$$
		dF_{\eta} : T_\eta\mathscr{A}^s \rightarrow H_{K, 0}^{s-1}(M)
		\quad \text{where} \;\;
		v \circ \eta \to d(\eta^*(\mathcal{L}_v \mu))/d\mu
		$$
		where $H_{K, 0}^{s-1}(M)$ is the space of $K$-invariant mean-zero functions on $M$
		and $\mathcal{L}_v$ is the Lie derivative in the direction of $v$.
		Since $\eta^*$ is an isomorphism of the space of smooth $n$-forms it suffices to show that
		$dF_e$ is a submersion on $\mathscr{A}_\mu^s$.
		Let $f = d\mathcal{L}_v\mu/d\mu$ with $\int_M f \, d\mu =0$ and
		using the Hodge decomposition\footnote{Cf. e.g., \cite{EbinMarsden}; Sect. 7.}
		choose a vector field $v=\nabla \Delta^{-1} f$ of Sobolev class $H^s$.
		Then\footnote{The musical symbols $\sharp$ and $\flat$ denote the usual isomorphisms between
		vector fields and differential forms induced by the Riemannian metric $g$ on $M$.}
		\begin{align} \nonumber
		[K, v]^\flat
		&=
		\mathcal{L}_K v^\flat - \mathcal{L}_K g (\cdot, v)
		\\ \label{eq:K-flat}
		&=
		d\iota_K v^\flat + \iota_K d v^\flat
		\\ \nonumber
		&=
		d\nabla_K \Delta^{-1} f
		=
		d\Delta^{-1} \langle\nabla f, K \rangle = 0
		\end{align}
		since $K$ is an infinitesimal isometry of $g$ and hence, in particular, it commutes with
		the Laplacian $\Delta$ and its inverse.
		This shows that $dF_e$ is surjective with split kernel (because $\mathscr{A}^s$ is a Hilbert manifold).
		
		That $\mathscr{A}_\mu^s$ is a topological group is an immediate consequence of the fact that
		the ambient space $\mathscr{D}^s_\mu(M)$ is a topological group whenever $s>5/2$.
		\end{proof}
	\begin{remark}
		Proposition $\ref{prop_axi_manifold}$ states that the space of diffeomorphisms that leave $K$ invariant
		is a smooth manifold modelled locally on the space of vector fields that commute with $K$.
		Note however that the latter may vary depending on the orbits of $K$.
		For example if $K = \partial_x + r\partial_y$ on $\mathbb{T}^3$ where $r$ is irrational then
		$\langle \nabla f, K \rangle = 0$ implies that $f$ is a function of the $z$ variable alone
		 while if $K = \partial_x$ then $f$ can be an arbitrary function of both variables $y$ and $z$.
		 In either case, we get a manifold structure on the space of diffeomorphisms preserving $K$
		 but without further assumptions on $K$ we cannot nice a generic description of this manifold
		(e.g., as parameterized by two functions of two variables).
	\end{remark}
	The next result shows that fluid flows with axisymmetric initial conditions remain axisymmetric.
	More precisely, we have
	\begin{thm} \label{prop_he_inv}
		If $u_0 \in T_e\mathscr{A}_\mu^s$ then the solution $u(t)$ of the Euler equations \eqref{eq:Euler}
		belongs to $T_e\mathscr{A}_\mu^s$ (as long as it is defined).
		In particular, $\mathscr{A}_\mu^s$ is a totally geodesic submanifold of $\mathscr{D}_\mu^s$.
	\end{thm}
	\begin{proof}
		The flow $\varphi_t$ of the Killing field $K$ on $M$ is  smooth and
		defined globally in time. The assumption $[K, u_0] = 0$ is equivalent to the fact
		that $u_0$ is constant along the flow, that is
		\begin{equation} \label{eq:u-0-constant}
		u_0 = \varphi_{t\ast\varphi_{-t}} u_0 \circ \varphi_{-t}
		\qquad \text{for all} \;\; t \in \mathbb{R}.
		\end{equation}

		For any fixed $t' \in \reals$ consider the vector field
		$$
		t \to \tilde{u}(t) = \varphi_{t'\ast\varphi_{-t'}} u(t)\circ\varphi_{-t'}.
		$$
		Since $\varphi_t$ is a flow of local isometries of $M$ the vector field $\tilde{u}(t)$ is divergence-free
		and we have
		\begin{align*}
		\partial_t\tilde{u} + \nabla_{\tilde{u}}\tilde{u}
		&=
		\varphi_{t'\ast\varphi_{-t'}} ( \partial_t{u} + \nabla_u u ) \circ \varphi_{-t'}
		\\
		&=
		\varphi_{t'\ast\varphi_{-t'}} ( -\nabla p \circ \varphi_{-t'} ) 
		\\ 
		&=
		-\nabla (p \circ \varphi_{-t'}).
		\end{align*}

		Furthermore, from \eqref{eq:u-0-constant} we also have $\tilde{u}(0)=u_0$ so that
		$\tilde{u}(t)$ and $u(t)$ must coincide by uniqueness of solutions to \eqref{eq:Euler}.
		Consequently, we find that
		$[K, u(t)]=0$, as long as $u(t)$ is defined.
		
		From Proposition \ref{prop_axi_manifold} it now follows that any $L^2$ geodesic in $\mathscr{D}^s_\mu$
		with initial velocity in $T_e\mathscr{A}_\mu^s$ remains in $\mathscr{A}_\mu^s$ which implies that
		$\mathscr{A}_\mu^s$ is totally geodesic.
	\end{proof}

	It follows that the restriction of the $L^2$ exponential map \eqref{eq:exp} to $T_e\mathscr{A}_\mu^s$
	yields a well-defined exponential map on $\mathscr{A}_\mu^s$ which we will continue to denote by $\exp_e$.
	Furthermore, its differential preserves the right-invariant distribution \eqref{eq_axi_tangent}
	in the sense that
	for any $u_0$ and $w_0$ in $T_e\mathscr{A}_\mu^s$ we have
	$d\exp_e(tu_0) tw_0 \in T_{\exp_e tu_0}\mathscr{A}_\mu^s$.

\subsection{Conservation of swirl}
\label{sec:conservation_swirl}

A general 3D axisymmetric fluid flow can display very complicated behavior.
Since vorticity is typically not conserved along particle trajectories,
the two-dimensional approach to global persistence fails unless additional conditions are imposed.
In the case of flows in $\mathbb{R}^3$ one such condition involves the swirl of the velocity field. The notion of swirl can be generalized to fluid flows in spaces of nonzero curvature.
Given an axisymmetric vector field $v$ on $M$ we define its \textit{swirl} to be the function
$\sigma_v = \langle v ,K \rangle$.
A vector field $v$ will be called \textit{swirl-free} if $\sigma_v =0$.

It turns out that, as in the Euclidean case, solutions of the Euler equations
which are initially swirl-free remain swirl-free as long as they exist.
In fact, the following more general result holds.
\begin{thm} \label{swirltransportthm}
The swirl of an axisymmetric velocity field is transported by its flow.
More precisely, if $u_0\in T_e\He$ and $\gamma(t)$ is the corresponding geodesic
in $\mathscr{A}_\mu^s$ then
$\sigma_{u(t)}\circ \gamma(t) = \sigma_{u_0}$
as long as it is defined.
\end{thm}
\begin{proof}
For simplicity, assume that $M$ is either a compact manifold with no boundary
or the Euclidean space $\mathbb{R}^3$.
Consider the function $f(t) = \langle u(t), K \rangle$.
Since $u(t)$ is a solution of the Euler equations \eqref{eq:Euler} we have
\begin{align*} \label{eq:trans}
\partial_t f + \nabla_u f
&=
\langle \partial_t u, K \rangle + \langle \nabla_uu, K \rangle + \langle u, \nabla_u K \rangle
\\ \nonumber
&=
- \langle \nabla p, K \rangle + \langle u, \nabla_u K \rangle
\end{align*}
where $p$ is the pressure of the fluid.
Since $K$ is a Killing field the second term on the right hand side is zero.

Furthermore, since $u$ is axisymmetric we have
\begin{align*}
\mathcal{L}_K \nabla_u u = \nabla_{[K, u]} u + \nabla_u [K, u] = 0
\end{align*}
and thus applying the Lie derivative in $K$ to \eqref{eq:Euler} and using the fact that
$\mathcal{L}_K g =0$, we obtain
\begin{align*}
d \langle \nabla p, K \rangle
=
\mathcal{L}_K dp
=
(\mathcal{L}_K \nabla p)^\flat
=
- \big( \partial_t \mathcal{L}_K u + \mathcal{L}_K \nabla_u u \big)^\flat
= 0.
\end{align*}
Since $M$ is connected it follows that $\langle \nabla p, K \rangle$ is a constant function
but since $K$ is divergence free (and tangent to the boundary if $\partial M \neq \emptyset$)
we also have
$\int_M \langle \nabla p, K \rangle \, d\mu =0$
which implies that the constant is zero.

Consequently, we find that $f$ satisfies a homogeneous transport equation with initial condition
$f(0)=\sigma_{u_0}$ and therefore has the form
$
f(t) = \sigma_{u_0} \circ \gamma^{-1}(t).
$
\end{proof}
%

\section{Axisymmetric diffeomorphisms with no swirl}
\label{sec:swirlfree}

In the previous section, we defined the notion of swirl for an axisymmetric vector field on $M$. It follows from this definition that the set of swirl-free vector fields is a linear subspace of the axisymmetric vector fields. It is natural to ask whether this subspace is also a Lie algebra, and if that is the case, whether there exists an underlying Lie group corresponding to it. The goal of this section is to provide a partial answer to these questions.

Let $T_e\Hel$ denote the space of all axisymmetric vector fields on $M$ with zero swirl. First, note that $T_e\Hel$ is a Lie subalgebra of $T_e\He$ if and only if the normal distribution
\begin{equation} \label{eq:K-int}
M {\setminus} \mathcal{Z} \ni x \to K^\perp_x = \big\{ V \in T_xM : \langle V, K(x) \rangle = 0 \big\}
\end{equation}
is integrable, where $\mathcal{Z}$ is the zero set of $K$. In Section \ref{sec_examples}, we provide examples illustrating the behavior of this distribution.

Under the integrability assumption we can decompose $M{\setminus}\mathcal{Z}$ into a disjoint union $\mathcal{F} = \bigcup_{x \in M{\setminus}\mathcal{Z}} \mathcal{F}_x$ of maximal leaves $\mathcal{F}_x$ which are tangent to the distribution $K^\perp$.

\begin{prop}\label{prop_flow_K}
The flow of $K$ preserves the foliation $\mathcal{F}$, that is
\begin{equation*}
\varphi_t \big( \mathcal{F}_x \big) \subseteq \mathcal{F}_{\varphi_{_t}(x)}
\end{equation*}
for any $x \in M {\setminus} \mathcal{Z}$ and any $t \geq 0$.
\end{prop}
\begin{proof}
Let $x \in M\backslash{}\mathcal{Z}$.
If $c(t')$ is a curve in $\mathcal{F}_x$ joining $x$ and some other point $y \in \mathcal{F}_x$
then $\varphi_t \circ c(t')$ is a curve joining $\varphi_t(x)$ and $\varphi_t (y)$ which is clearly orthogonal to $K$.
Thus $\varphi_t(y) \in \mathcal{F}_{\varphi_t(x)}$.
\end{proof}
An axisymmetric diffeomorphism $\eta$ of $M$ will be called \emph{swirl-free}
if it maps each leaf of $\mathcal{F}$ to itself; i.e.
$\eta(\mathcal{F}_x) \subseteq \mathcal{F}_x$ for every $x \in M {\setminus} \mathcal{Z}$.

\par To simplify the exposition we will assume from now on that the flow of $K$ is transitive on the leaves
in the sense that for any $x$ and $y$ in $M{\setminus}\mathcal{Z}$ there is a $t \in \reals$ such that
$\varphi_t(\mathcal{F}_x) = \mathcal{F}_y$. This transitivity implies that an axisymmetric diffeomorphism $\eta$ that fixes a single leaf (setwise) must also fix all the other leaves.
To see this, let $x_0 \in M{\setminus}\mathcal{Z}$ and suppose that
$\eta(\mathcal{F}_{x_0}) \subseteq \mathcal{F}_{x_0}$.
Then, given any $x \in M{\setminus}\mathcal{Z}$ we can choose $t$
with $\varphi_t(\mathcal{F}_{x_0}) = \mathcal{F}_x$ and find that
$$
\eta(\mathcal{F}_x)
=
\eta \circ \varphi_t (\mathcal{F}_{x_0})
=
\varphi_t \circ \eta ( \mathcal{F}_{x_0} )
\subseteq
\varphi_t ( \mathcal{F}_{x_0} )
=
\mathcal{F}_x.
$$

To state the main theorem in this section, we will also require that 
any sufficiently short geodesic which starts at a point $p$ of a given leaf $\mathcal{F}_x$ and is transversal to $\mathcal{F}_x$ 
does not return to $\mathcal{F}_x$ (see theorem \ref{thm_swirlfreesubgroup} below for a precise definition).

This condition is clearly satisfied in the standard Euclidean example of $M=\mathbb{R}^3$ with rotational symmetry,
i.e., $K = \partial_{\theta}$. In this case the leaves of the corresponding foliation $\mathcal{F}$
are the half-planes in $\mathbb{R}^3$ defined by $\theta = \mathrm{const.}$ (in polar coordinates)
and a geodesic starting transversally from any point on a given half-plane never returns to the same half-plane.
\begin{theorem} \label{thm_swirlfreesubgroup}
Assume that the foliation $\mathcal{F}$ is integrable and suppose that
for some $x_0 \in M{\setminus}\mathcal{Z}$ the leaf $\mathcal{F}_{x_0}$ is a submanifold of $M$
that admits an $\epsilon$-tubular neighborhood.\footnote{Recall that a submanifold $N$ of a Riemannian manifold
$M$ admits an $\epsilon$-tubular neighborhood if for any $p \in N$ and $v \in T_pM$ with $\|v\| < \epsilon$
we have $\exp_p v \in N$ if and only if $v \in T_pN$.}
Then the set of all swirl-free Sobolev $H^s$ diffeomorphisms
\begin{equation*}
\Hel = \big\{ \eta \in \He : \eta(\mathcal{F}_x) \subseteq \mathcal{F}_x, \; x \in M{\setminus}\mathcal{Z} \big\}
\end{equation*}
is a topological group and a smooth Hilbert submanifold of $\He$.
\end{theorem}
\begin{proof}
The proof follows \cite[Theorem 6.1]{EbinMarsden} with obvious modifications.
\end{proof}
\begin{cor} \label{cor:swirlfreesubgroup}
In particular, $\Hel$ is a totally geodesic submanifold $\He$ (and of $\D$).
\end{cor}
\begin{proof}
This follows directly from Theorems \ref{prop_he_inv} and \ref{swirltransportthm}.
\end{proof}
%

\subsection{The diameter of $\mathscr{A}^s_{\mu,0}$}

In \cite{ShnVol} Shnirelman proved that the group of volume-preserving diffeomorphisms of
a compact simply-connected manifold of dimension greater than two has finite $L^2$ diameter
(see also \cite{Shn94}, \cite{AK}).
On the other hand, Eliashberg and Ratiu \cite{ElRat} showed that the $L^2$ diameter of
the symplectomorphism group of any compact exact symplectic manifold (which in dimension two
coincides with the group of volume-preserving diffeomorphisms) is infinite. Their proof relies on the fact that
it is possible to construct exact symplectomorphisms of such manifolds that are fixed near the boundary
and whose Calabi invariants take on arbitrarily large values.

It turns out that their construction can be adapted to the subgroup of swirl-free diffeomorphisms.
\begin{thm}
The $L^2$ diameter of $\Hel$ is infinite.
\end{thm}
\begin{proof}
Consider the case when the underlying manifold is a cylinder of finite height in $\mathbb{R}^3$,
namely
$$
M = \big\{ (r, \theta, z): \, 0 \leq r, z \leq 1, \, 0 \leq \theta <2\pi \big\}
$$
where $(r, \theta, z)$ are the cylindrical coordinates and the Killing field is $K = \partial_\theta$.
In this case Theorems \ref{prop_he_inv}, \ref{thm_swirlfreesubgroup} and Corollary \ref{cor:swirlfreesubgroup}
imply the following inclusions of totally geodesic submanifolds
\begin{equation*}
\Hel \subset \He \subset \mathscr{D}^s_{\mu}.
\end{equation*}
It will be sufficient to show that the diameter of $\Hel$ computed in the right invariant metric given by
the $L^1$ norm on the algebra of vector fields is infinite. The result for the $L^2$ diameter will then follow
from H\"older's inequality.

Let $\eta(t)$ be any path connecting the identity element $e=\eta(0)$ with some $\eta_1=\eta(1)$ in $\mathscr{A}_{\mu,0}^s$.
Since the diffeomorphisms along the path are axisymmetric we have
\begin{align} \label{eq:L1_length}
\int_0^1 \int_M  \big| \dot\eta_{t}(r, \theta, z) \big| \, r drd\theta dz dt
=
2\pi \int_0^1 \int_{Q^2} \big| \dot\eta_{t}(r, z) \big| \, r dr dz dt
\end{align}
where $Q^2$ is the two-dimensional square "slice" of the cylinder, that is
$$
Q^2 = \big\{ (r, 0, z) : 0 \leq r, z \leq 1 \big\} \subset M.
$$
It follows from \eqref{eq:L1_length} that the $L^1$ length of $\eta(t)$ in $\Hel$ is bounded from below
by its $L^1$ length as a curve in $\mathcal{D}^s_{\mu}(Q^2)$.
Observe that even though $Q^2$ has corners we can work with diffeomorphisms which fix
a neighbourhood of the boundary.
Recall also that the Calabi invariant defines an epimorphism from the group of exact symplectomorphisms of $Q^2$
which are fixed near the boundary to $\mathbb{R}$, cf. e.g., \cite{AK}.

The rest of the argument proceeds now as follows.
First, given any $R > 0$ cut out an open disk $D$ in the interior of $Q^2$ and
construct a symplectomorphism $\xi_R : D \rightarrow D$ whose Calabi invariant is greater than $R$
and $\xi_R = \mathrm{id}$ near the boundary $\partial D$, as in \cite{ElRat}; Lem.~6.1.
Next, extend $\xi_R$ to the identity map on $Q^2\backslash D$
and, finally, to an axisymmetric diffeomorphism of $M$ by composing
$\rho_\theta \circ \xi_R (r, 0, z)$, where $\rho_\theta$ is a rotation by $0 \leq \theta < 2\pi$ of $\mathbb{R}^3$.

Observe that any path $\eta(t)$ in $\Hel$ starting at the identity with $\eta_1 = \rho_\theta \circ \xi_R$
will have its $L^1$ length greater than $2\pi R C$,
where $C>0$ is a constant depending only on $M$.
This shows that $\Hel$ has infinite $L^2$ diameter.
\end{proof}
\begin{remark} \label{rem:diameter}
It is interesting to observe that the $L^2$ diameter of $\Hel$ is infinite even though it is a totally geodesic submanifold of $\He$ whose diameter is finite by Shnirelman's result. Finite dimensional examples of this situation are the following.
Let $\mathbb{T}^n$ be the $n$-dimensional flat torus and $N \subseteq \mathbb{T}^n$ be any $k$-dimensional totally geodesic submanifold such that $N$ is dense in $\mathbb{T}^n$ (i.e., a higher dimensional analogue of the standard ``irrational flow'' example). Then $\mathbb{T}^n$ has finite diameter, but $N$ is isometric to $\mathbb{R}^k$ and therefore has infinite diameter.
%

On the other hand, if a compact manifold has positive sectional curvature then so does any totally geodesic submanifold
(of dimension greater than 1) and therefore it must also be compact by the Bonnet-Myers theorem.
\end{remark}
%

\section{Fredholm properties of the $L^2$ exponential map}
\label{section_fred}

It turns out that the $L^2$ exponential map in 3D hydrodynamics is Fredholm provided that
it is restricted to axisymmetric flows of sufficiently small swirl (including the swirl-free flows).
Thus, in this case the situation resembles that of 2D hydrodynamics.
We state our main results for the derivative of the exponential map because, as we have seen, in some cases
there is no guarantee that such flows are confined to a smooth submanifold of the diffeomorphism group
(cf. Sections \ref{sec:ax_diff}, \ref{sec:swirlfree}).
We emphasize that the vanishing conditions on the swirl are necessary.
For example, in \cite{EMP06} the authors show that if $u_0$ is a rigid rotation of a solid cylinder in $\mathbb{R}^3$
then the $L^2$ geodesic in the direction of $u_0$ (which lies in $\He$) contains clusters of conjugate points
along segments of finite length
- in particular, the associated exponential map cannot be Fredholm.

The next lemma will be crucial in what follows. It shows that the curl of an axisymmetric swirl-free vector field
has a very special form. We will assume that $K$ either has no zeroes or that they are of the same form as in the usual axisymmetric case: rotation along an axis in a warped product metric. The following metric seems very special, but in fact is general enough to cover all the standard three-dimensional geometries, as we shall see later.

\begin{deff}\label{rotationlike}
A Killing field $K$ on a Riemannian manifold $M$ is called a \emph{rotation} if there is a coordinate chart $(r,\theta,z)$ such that $K=\frac{\partial}{\partial \theta}$ where the metric is of the form
\begin{equation}\label{rotationmetric}
ds^2 = dr^2 + \alpha(r)^2 \, d\psi^2 + \beta(r)^2 \, d\theta^2
\end{equation}
where $\theta\in S^1$ and $\beta$ extends to a $C^{\infty}$ odd function through $r=0$ satisfying $\beta'(0)=1$, and $\beta(0)=1$.
\end{deff}

Note that the smoothness requirement on $\beta$ is a general feature of Riemannian metrics as discussed in Petersen, Chapter 1, Section 3.4~\cite{petersen}. Zeroes of Killing fields on three-dimensional manifolds must occur along geodesics by a theorem of Kobayashi~\cite{kobayashi}, but we will not consider the general situation since the family \eqref{rotationmetric} is rich enough to give
many examples while still yielding relatively simple equations.
\begin{lem} \label{lem_concentration}
Suppose $K$ is nowhere zero or a rotation. If $v \in T_e\mathscr{A}_{\mu,0}^{s+1}$ then $\mathrm{curl}\, v = \phi K$ where $\phi$ is a function of class $H^s$.
\end{lem}
\begin{proof}
First, assume that the Killing field $K$ does not vanish.
Given any $x \in M$ choose smooth vector fields $E_1, E_2$ such that together with $K/|K|$
they define a local coordinate frame of orthonormal fields at $x$.
The axisymmetric assumption on $v$ and the fact that $K$ is Killing imply that $\mathcal{L}_K v^\flat = 0$
(cf. \eqref{eq:K-flat})
and from Cartan's formula we have
\begin{equation} \label{eq_cartan}
0 = \mathcal{L}_K v^\flat = d\iota_K v^\flat + \iota_K dv^\flat = \iota_K dv^\flat
\end{equation}
since $v$ is also swirl-free by assumption.
Let $\omega_1, \omega_2$ and $\omega_3$ be the component functions of Sobolev class $H^s$
of the associated 2-form near $x$, so that
\begin{equation*}
dv^\flat
=
\omega_3 E_1^\flat \wedge E_2^\flat
+
\omega_2 K^\flat \wedge E_1^\flat
+
\omega_1 E_2^\flat \wedge K^\flat.
\end{equation*}
Evaluating on the frame fields and using $\eqref{eq_cartan}$ we now have
\begin{gather*}
\omega_2 |K|^2
=
dv^\flat(K, E_1)
=
0
=
dv^\flat(E_2, K)
=
\omega_1 |K|^2
\end{gather*}
which implies
$dv^\flat = \omega_3 E_1^\flat \wedge E_2^\flat$
and consequently
$\mathrm{curl}\, v = (\ast dv^\flat)^\sharp = \phi K$
for some $H^s$ function $\phi$ and where $\ast$ is the Hodge star operator.
This completes the proof in the non-vanishing case.

Now if $K$ is a rotation, then we can use more concrete coordinates: we have $E_1 = \partial_r$
and $E_2 = \frac{1}{\alpha(r)} \, \partial_{\psi}$, and a divergence-free, swirl-free field $v$ can be
written locally as $$v = -\frac{1}{\alpha(r)\beta(r)} \, \frac{\partial f}{\partial \psi} \, \partial_r
+ \frac{1}{\alpha(r)\beta(r)} \, \frac{\partial f}{\partial r} \, \partial_{\psi},$$
for an $H^{s+2}$ stream function $f$, which must extend to an even function of $r$ in order for the radial component
of $v$ to vanish along the axis. Then the curl of $v$ is given by $\curl{v} = \phi K$ where
$$ 
\phi \, \frac{1}{\alpha(r)\beta(r)} \, \frac{\partial}{\partial r} \left( \alpha(r)\beta(r) \, \frac{\partial f}{\partial r} \right)
+ 
\frac{1}{\alpha(r)^2 \beta(r)^2} \, \frac{\partial^2 f}{\partial z^2}. 
$$
Since $\alpha(r)$ and $\frac{\beta(r)}{r}$ are both smooth even functions of $r$ which approach $1$ as $r\to 0$, 
we see that this differential operator behaves the same way as in the Euclidean case where $\alpha(r)=1$ 
and $\beta(r)=r$. 
In particular $\phi$ is $H^s$ as in the Euclidean case, cf. e.g., \cite{Danchin} or \cite{MajdaBertozzi}.
\end{proof}
%

An immediate consequence of the above lemma and the conservation of swirl of
Section \ref{sec:conservation_swirl} is a global existence result for swirl-free flows in any curved space.
Before formulating the result we first recall a special case of particular interest.
Consider the Euclidean space $\mathbb{R}^3$ with cylindrical coordinates $(r,\theta,z)$
and let the Killing field be $K = \partial_{\theta} + \kappa\partial_z$, where $\kappa \in \mathbb{R}$.
The case $\kappa = 0$ corresponds to the standard rotational case and it is not difficult to check that
the quantity $\omega^\theta /r$, where $\omega^\theta$ is the $\theta$-component of the vorticity field,
is constant in time.
The case $\kappa \neq 0$ is sometimes described as helicoidal symmetry
and the corresponding conserved quantity is $\omega^\theta /\sqrt{r^2 + \kappa^2}$.
In both of these cases global existence and uniqueness results are well-known,
see e.g., \cite{uy}, \cite{MajdaBertozzi}, \cite{Dutrifoy} or \cite{Danchin}.

In the general we have
\begin{theorem}
If $u_0 \in T_e\mathscr{A}_{\mu,0}^s$ then the corresponding solution $u(t)$ of
the Euler equations $\eqref{eq:Euler}$ in $M$ can be extended globally in time.
\end{theorem}
\begin{proof}
We assume that the Killing field $K$ does not vanish.
Since by assumption $u_0$ is swirl-free, it follows from Theorem \ref{swirltransportthm} that
the swirl $\sigma_{u(t)}=0$ as long as $u$ is defined.
But in this case Lemma \ref{lem_concentration} implies that the vorticity of $u(t)$ has the form
$\omega(t,x) = \phi(t,x) K(x)$ where $\phi=\phi(t,x)$ is a time-dependent function of class $H^{s-1}(M)$.
Therefore, using the 3D vorticity equation and the fact that $u(t)$ is axisymmetric,
we have
\begin{align*}
(\partial_t \phi + \nabla_u \phi ) K
&=
\partial_t (\phi K) + \nabla_u (\phi K) - \phi \nabla_u K
\\
&=
\partial_t \omega + \nabla_u \omega - \nabla_\omega u
\\
&= 0.
\end{align*}
This shows that the function $\phi$ is also transported along the flow of $u(t)$.
Consequently, the $L^\infty$ norm of $\omega$ is constant in time and the theorem follows now
from the Beale-Kato-Majda criterion, see e.g., \cite{MajdaBertozzi}.
\end{proof}

\subsection{The case of no swirl}

In this subsection we study Fredholm properties of the exponential map on $\He$ in the swirl-free directions.
Our first result is contained in the following theorem.
\begin{thm} \label{fredholmness}
Let $M$ be either a compact Riemannian 3-manifold or $\mathbb{R}^3$ equipped with a smooth Killing field $K$.
Let $u_0 \in T_e\mathscr{A}_{\mu,0}$ be an axisymmetric swirl-free vector field of Schwartz class.
Then the derivative
$$
{d}\exp_e(u_0) \colon T_e\He \rightarrow T_{\exp_e{u_{_0}}} \He
$$
is a Fredholm operator of index zero.
\end{thm}
\begin{remark}
If the assumptions of Section \ref{sec:swirlfree} are satisfied then
(with some extra work needed to relax the smoothness assumption in order to allow
$u_0 \in T_e\mathscr{A}^s_{\mu,0}$ for any $s>5/2$, cf. Remark \ref{rem:modulo} above)
it follows from the above result and Theorem \ref{thm_swirlfreesubgroup} that the $L^2$ exponential map in \eqref{eq:exp}
when restricted to the subgroup $\mathscr{A}_{\mu,0}^s \subset \mathscr{A}_\mu^s$
of swirl-free diffeomorphisms
$$
\exp_e : T_e\mathscr{A}^s_{\mu,0} \to \mathscr{A}^s_{\mu,0}
$$
is a nonlinear Fredholm map of index zero.
\end{remark}
\begin{proof}[Proof of Theorem \ref{fredholmness}]
As before, we first assume that $M$ is compact and the Killing field $K$ does not vanish.
To further simplify our calculations we will also assume that the first homology group of $M$ is zero.\footnote{In this case
the group of volume-preserving diffeomorphisms $\mathscr{D}_\mu^s$ coincides with the group of
exact volumorphisms $\mathscr{D}_{\mu,ex}^s$ of $M$.}
Since on a three-dimensional manifold a 2-form can be identified with a 1-form using Hodge duality,
it follows that any $v \in T_e\mathscr{D}_\mu^s$ can be now written as
\begin{equation} \label{eq:curls}
v
=
( \delta d\beta )^\sharp
=
( \delta\ast w^\flat )^\sharp
=
\mathrm{curl}\, w
\end{equation}
for some 1-form $\beta$ of Sobolev class $H^{s+2}$ and $w^\flat = \ast d\beta$.
\begin{lem} \label{lem:Ads}
The Lie group adjoint and coadjoint operators on the volumorphism group $\mathscr{D}_\mu^s$
of a 3-manifold
are given by the formulas
\begin{align*}
\mathrm{Ad}_\eta v
=
\eta_\ast v
=
\mathrm{curl} ( \eta_\ast w^\flat )^\sharp
\quad \text{and} \quad
\mathrm{Ad}_\eta^\ast v
=
\mathrm{curl} \, \Delta^{-1} \eta^{-1}_\ast \Delta w
\end{align*}
where $v \in T_e\mathscr{D}_\mu^s$ and $w \in T_e\mathscr{D}_\mu^{s+1}$ are related by \eqref{eq:curls}.
The corresponding Lie algebra coadjoint operator is given by
$$
\mathrm{ad}^\ast_v u
=
\mathrm{curl}\, \Delta^{-1} [v, \mathrm{curl}\, u]
$$
for any $u, v \in T_e\mathscr{D}_\mu^s$.
\end{lem}
\begin{proof}
Both formulas can be found in \cite{MP10}; Prop. 2.5, Prop. 3.6.
\end{proof}

Let $\gamma(t)$ be the $L^2$ geodesic starting from the identity $e$ in the direction $u_0$.
Consider the family of bounded operators $\Phi_t$ on $T_e\mathscr{D}^s_\mu$ defined in \eqref{eq:Phi}
with the attendant decomposition
in \eqref{eq:solop} of Proposition \ref{prop:OLD}.
From Section \ref{sec:ax_diff} we know that $d\exp_e(tu_0)$ preserves the tangent spaces to $\He$
so that each $\Phi_t$ is well-defined as a bounded operator on $T_e\He$.

Regarding the first term in \eqref{eq:solop} we have
\begin{lem} \label{lem_inv}
The operators $\Omega_t$ are invertible on $T_e\He$.
\end{lem}
\begin{proof}[Proof of Lemma $\ref{lem_inv}$]
The proof that the operators $\Omega_t$ in \eqref{eq:Omega} are invertible on $T_e\mathscr{D}^s_\mu$
can be found in \cite{EMP06}, Prop. 7.
In particular, each $\Omega_t$ is bounded when restricted to the subspace $T_e\He$.

Observe that if $\eta\in \mathscr{A}^{s+1}_\mu$ then the product
$\Ad_{\eta^{-1}}\Ad_{\eta^{-1}}^{*}$
of the adjoint operator and its $L^2$ coadjoint maps $T_e\mathscr{A}^s$ to itself.
In fact, if $v \in T_e\mathscr{A}_\mu^s$ then
$w^\flat = \ast \Delta^{-1} dv^\flat$
and from Lemma \ref{lem:Ads} it follows at once that both factors map into $H^s$ and that
$\mathrm{div} (\mathrm{Ad}_{\eta^{-1}}^\ast v ) = \mathrm{div} ( \mathrm{Ad}_{\eta^{-1}} v ) = 0$.
Furthermore, we have
\begin{align} \label{eq:1}
[ K, \mathrm{Ad}_{\eta^{-1}}^\ast v ]
&=
\mathcal{L}_K \mathrm{Ad}_{\eta^{-1}}^\ast v
=
\mathcal{L}_K \mathrm{curl}\, \Delta^{-1} \eta^{-1}_\ast \Delta w
\\ \nonumber
&=
\mathrm{curl}\, \Delta^{-1} \mathcal{L}_K \eta^{-1}_\ast \Delta w
=
\mathrm{curl}\, \Delta^{-1} \eta^{-1}_\ast \Delta \mathcal{L}_K w
=
0
\end{align}
since $\eta$ being axisymmetric implies $\eta^{-1}_\ast K=K$ and the Lie derivative is natural
with respect to the push-forward map.
The last step in \eqref{eq:1} follows again from the fact that $K$ is an infinitesimal isometry so that
\begin{align*}
(\mathcal{L}_K w)^\flat
=
\mathcal{L}_K w^\flat
=
\ast \Delta^{-1} d \mathcal{L}_K v^\flat
=
0
\end{align*}
by the axisymmetry assumption on $v$.
Therefore $\mathrm{Ad}_{\eta^{-1}}^\ast v$ lies in the subspace $T_e\mathscr{A}_\mu^s$.
Similarly, we have
\begin{align*}
[K, \mathrm{Ad}_{\eta^{-1}} v ]
=
\mathcal{L}_K \eta^{-1}_\ast v
=
\eta^{-1}_\ast \mathcal{L}_K v
= 0
\end{align*}
and hence $\mathrm{Ad}_{\eta^{-1}} v$ is also in $T_e\mathscr{A}_\mu^s$.

Since $u_0 \in T_e\mathscr{A}_{\mu,0}$ the corresponding $L^2$ geodesic $\gamma(t)$ is smooth.
It follows that each operator
$
\Omega_t
=
\int_{0}^{t}\mathrm{Ad}_{\gamma_\tau^{-1}}\mathrm{Ad}_{\gamma_\tau^{-1}}^{*} d\tau
$
is a continuous linear bijection of $T_e\He$ to itself and hence an isomorphism by Banach's theorem.
\end{proof}

Our key observation concerning the operators $\Gamma_t$ in \eqref{eq:Gamma} is contained
in the following lemma.
\begin{lem} \label{compactness_lemma}
The restriction of the algebra coadjoint
$v \to \mathrm{ad}^\ast_v u_0$ to $T_e\He$ is a compact operator.
\end{lem}
\begin{proof}
Since $u_0 \in T_e\mathscr{A}_{\mu,0}$ is swirl free, it follows from Lemma \ref{lem_concentration} that
$\mathrm{curl}\, u_0 = \phi K$ for some function $\phi$ of Schwartz class.
Using Lemma \ref{lem:Ads} we then have
\begin{equation} \label{KK}
v \mapsto \mathrm{ad}^\ast_v u_0
=
\mathrm{curl} \,\Delta^{-1} [v, \mathrm{curl}\,u_0 ]
=
\mathrm{curl} \,\Delta^{-1} \big( d\phi(v) K \big)
\end{equation}
since $v$ is an axisymmetric vector field.
The result follows directly from the usual Rellich-Kondrashov lemma.
\end{proof}

As a consequence of Lemma \ref{compactness_lemma} and Proposition \ref{prop:OLD}
we find that each operator
$
\Gamma_t (\cdot)
=
\int_0^t
\mathrm{Ad}_{\gamma_\tau^{-1}} \mathrm{Ad}_{\gamma_\tau^{-1}}^\ast \mathrm{ad}_{\Phi_\tau (\cdot)}^\ast u_0 \, d\tau
$
is compact as an integral (in $t$) of a product of bounded linear operators and a compact operator.
Combined with Lemma \ref{lem_inv} and \eqref{eq:solop} this implies that
each $\Phi_t$ decomposes on $T_e\mathscr{A}_{\mu}^s$ into a sum of an invertible and a compact operators
and is therefore a bounded Fredholm operator whose index equals
$\mathrm{ind}\, \Phi_t = \mathrm{ind} \, \Phi_0 = 0$
by a standard perturbation argument.
\begin{remark}
If the first homology $\mathrm{H}^1(M) \neq 0$ then the expression for $\mathrm{ad}^\ast_v u_0$
would differ from that in \eqref{KK} by an operator of finite rank. Therefore the Lie algebra coadjoint
would also be compact in that case and the rest of the preceding argument could be adjusted accordingly.
\end{remark}

We next consider the case $M = \mathbb{R}^3$ with rotational symmetry, that is, $K=\partial_\theta$,
and to simplify the argument we first assume that $u_0$ has compact support in $\mathbb{R}^3$.
The formulae of Lemma \ref{lem:Ads} which rely on the Hodge decomposition (which also holds in $\mathbb{R}^3$)
remain unchanged and so does Lemma \ref{lem_inv}.
We focus therefore on compactness of the coadjoint operator in $H^s(\mathbb{R}^3)$.
As before, we may disregard the finite rank terms and therefore we only need to show that
the operator in \eqref{KK} maps bounded sets to relatively compact sets.

From \eqref{KK} using Plancherel's inequality, standard potential theory estimates
and the fact that $\phi$ has compact support we have
\begin{align*} \nonumber
\big\| \mathrm{curl}\, \Delta^{-1} \big( d\phi (v) K \big) \big\|_{H^s(\mathbb{R}^3)}
&\lesssim
\| d\phi (v) K \|_{L^{6/5}(\mathbb{R}^3)}
+
\| d\phi (v) K \|_{\dot{H}^{s-1}(\mathbb{R}^3)}
\\ \label{eq:ad-op}
&\lesssim
\| d\phi (v) K \|_{L^2(\mathbb{R}^3)}
+
\| d\phi (v) K \|_{\dot{H}^{s-1}(\mathbb{R}^3)}
\\ \nonumber
&\simeq
\| d\phi (v) \|_{H^{s-1}(\mathbb{R}^3)}
\end{align*}
where the (suppressed) constants possibly depend on $\phi$.
The above estimate implies that the coadjoint $\mathrm{ad}^\ast u_0$ is bounded as an operator to $H^{s-1}$
and so the fact that it is compact follows once again from the Rellich-Kondrashov lemma.

To dispose of the more general case when $u_0$ is of Schwartz class one can use a method of
approximating in $H^s$ by functions with compact support.
This completes the proof of Theorem \ref{fredholmness}.
\end{proof}
\begin{remark}
In Section $\ref{sec:ex}$ we will derive a more explicit expression
for the coadjoint operator based on a decomposition of axisymmetric vector fields
into gradient and swirl components (cf. Proposition $\ref{axisymmetricvectorfieldclassification}$).
It can also be used to give an alternative proof of Lemma \ref{compactness_lemma}.
\end{remark}
%
%
%
%

\subsection{A perturbation result}

Using a straightforward argument it is possible to extend Theorem \ref{fredholmness}
to more general axisymmetric flows provided that they have suitably small swirl.
\begin{thm} \label{fredholmness2}
Let $u_0$ be a compactly supported, smooth, axisymmetric and swirl-free vector field on $M$.
For any $v \in T_e\mathscr{A}_\mu^s$ which is sufficiently close to $u_0$ in the $H^s$ norm
the derivative
$$
d\exp_e(v) : T_e\mathscr{A}^s \rightarrow T_{\exp_e(u_0)}\mathscr{A}^s
$$
is a Fredholm operator of index zero.
\end{thm}
\begin{proof}
Along any geodesic $\gamma(t)$ in $\mathscr{D}_\mu^s$ there is an operator $\tau_t^\gamma$
of parallel translation with respect to the $L^2$ metric \eqref{eq:L2}.
As in finite dimensions, it is defined as the solution operator of the Cauchy problem
for the corresponding first-order ordinary differential equation in $\mathscr{D}_\mu^s$
$$
X'(t) = 0, \quad X(0) \in T_e\mathscr{D}_\mu^s
$$
where $'$ is the Levi-Civita covariant derivative of the $L^2$ metric along $\gamma$.
Since the Levi-Civita derivative restricted to the submanifold $\mathscr{A}_\mu^s$ is smooth,
it follows that the operator $\tau_t^\gamma$ is a linear isomorphism of the tangent spaces
to $\mathscr{A}_\mu^s$ as well as an isometry of the metric \eqref{eq:L2}.
We denote its inverse by $\tau_{-t}^\gamma$.

Let $\gamma_{u_0}$ and $\gamma_v$ be the $L^2$ geodesics in $\mathscr{A}_\mu^s$
starting from the identity in the direction $u_0$ and $v$, respectively.
Assume that $\gamma_v(t)$ is defined at least for $0 \leq t \leq 1$.
Consider the map
$v \to \tau^{\gamma_v}_{-1} \circ d\exp_e(v)$
and observe that it is smooth in the $H^s$ topology as a composition of two smooth maps.
The latter follows from the $C^\infty$ regularity of the $L^2$ exponential map on $\mathscr{A}_\mu^s$
and the smooth dependence on parameters of the parallel translation operator
along geodesics in $\mathscr{A}_\mu^s$.
Consequently, we can find a $\delta >0$ such that whenever $\|v\|_{H^s} < \delta$ then
\begin{equation*}
\big\| \tau^{\gamma_{u_0}}_{-1} \circ d\exp_e(u_0) - \tau^{\gamma_v}_{-1} \circ d\exp_e(v) \big\|_{L(H^s)}
\lesssim
\|v - u_0\|_{H^s}
\end{equation*}
by a mean-value estimate.
The theorem now follows from stability results for Fredholm operators,
see e.g., \cite{Kato}; Thm. 5.16.
\end{proof}
\begin{remark}[$L^2$ sectional curvature]
In \cite{PrWash} the authors studied an axisymmetric flow in the solid torus $D^2 \times \mathbb{S}^1$
with initial condition of the form $u_0(r, z) = f(r)\partial_{\theta}$, where $f$ is a smooth function of the radial variable.
They showed that the sectional curvatures of $\mathscr{A}^s_\mu$ were strictly positive
on the planes spanned by $u$ and $v$ if and only if $\frac{d}{dr}(rf^2) > 0$,
where $u$ is the solution of the Euler equations \eqref{eq:Euler} with data $u_0$
and $v$ is any axisymmetric vector field on the torus.
Furthermore, if $f(2f+rf') > 0$ then the $L^2$ geodesic in $\mathscr{A}_\mu^s$
corresponding to $u_0$ contains clusters of conjugate points and, in particular, the $L^2$ exponential map
is not Fredholm.
Observe that their results do not contradict Theorems \ref{fredholmness} and \ref{fredholmness2} since
the solution of \eqref{eq:Euler} corresponding to $u_0$ above necessarily develops a large swirl
in finite time.

We may view this as a more precise version of the general result from Ebin-Marsden~\cite{EbinMarsden} that the exponential map is
locally a diffeomorphism for small velocity fields, which follows automatically from the smoothness of the geodesic equation.
Here the velocity field does not need to be small in $H^s$, but rather its swirl just needs to be small in $H^{s-1}$
so that $v$ is close to a swirl-free field $u_0$.

On the other hand, choosing a swirl-free initial data of the form $u_0(r,z) = f(r)\partial_z$,
where $f$ is any smooth function of the radial variable, we find that
the corresponding Eulerian solution $u$ satisfies $\nabla_u u = 0$. In this case the Gauss-Codazzi equations
of submanifold geometry (see e.g., \cite{Mi92} or \cite{MP10}) applied to the $L^2$ sectional curvatures of
$\mathscr{A}_\mu^s$ and $\mathscr{D}^s$ yield
\begin{align} \nonumber
\left\langle \mathcal{R}(u, v)v, u \right\rangle_{L^2}
&=
\left\langle \bar{\mathcal{R}}(u, v)v, u \right\rangle_{L^2}
+
\left\langle Q_{e} \nabla_uu,~ Q_{e} \nabla_vv \right\rangle_{L^2}
-
\left\langle Q_{e} \nabla_uv,~ Q_{e} \nabla_uv \right\rangle_{L^2}
\\
&=
-
\left\langle Q_{e} \nabla_uv,~ Q_{e} \nabla_uv \right\rangle_{L^2}
\leq 0
\end{align}
for any axisymmetric vector field $v$.
Here $Q_e = \nabla \Delta^{-1} \mathrm{div}$ is the Hodge projection onto the gradient fields
and $\mathcal{R}$, $\bar{\mathcal{R}}$ are the Riemann curvature tensors of the $L^2$ metric
on $\mathscr{D}_\mu^s$, $\mathscr{D}^s$, respectively,
and in the second line we used the fact that the $L^2$ curvature of the full diffeomorphism group
of $D^2\times\mathbb{S}^1$ is zero.
That the $L^2$ sectional curvatures along this flow are all non-positive implies that the corresponding geodesic
in the totally geodesic submanifold $\mathscr{A}_\mu^s$ is free of conjugate points.
In particular, in this case the $L^2$ exponential map is Fredholm.
\end{remark}
%

\subsection{Exponential map and coadjoint orbits}

It is well known that the geometry of coadjoint orbits of the group of volume-preserving diffeomorphisms
of a 3-manifold is considerably more complicated than that of a 2-manifold.
Nevertheless an elegant geometric description of the orbits was found by Arnold \cite{Ar11}
in terms of isovorticity of vector fields.\footnote{In our setting, we say that two vector fields
$v_1, v_2$ are \textit{isovortical} if there is an axisymmetric diffeomorphism $\eta$
such that the circulation of $v_1$ around any closed contour $c$ in $M$
and the circulation of $v_2$ around $\eta \circ c$ are equal.}
In particular, if two axisymmetric vector fields $v$ and $\tilde v$ are isovortical
then the corresponding vorticity fields satisfy
\begin{equation} \label{eq:isovort}
\mathrm{curl}\, v = \eta_\ast \mathrm{curl}\, \tilde v
\end{equation}
for some axisymmetric diffeomorphism $\eta$.
Observe that, in general, vector fields that are isovortical to a swirl-free vector field need not be themselves swirl-free.
A simple example is given by the vector fields
$$
v = \partial_z
\quad \text{and} \quad
\tilde v = r^{-2} \partial_\theta
$$
on a finite (hollow) cylinder with periodic conditions
$$
M = \big\{ (r,\theta, z): 0 < a \leq r \leq b <\infty, \, 0 \leq \theta < 2\pi, \, -\pi \leq z < \pi \big\}
$$
in $\mathbb{R}^3$ with cylindrical coordinates and rotational symmetry.

As a direct consequence we have
\begin{thm} \label{fred-coadjoint}
Let $M$ be a compact Riemannian 3-manifold with a smooth Killing field $K$.
If $u_0$ and $\tilde{u}_0$ are related as in \eqref{eq:isovort} with $u_0 \in T_e\mathscr{A}_{\mu, 0}$
and some $\eta \in \mathscr{A}_{\mu}$, then the derivative
$$
d\exp_e(\tilde{u}_0) \colon T_e\mathscr{A}_\mu^s \rightarrow T_{\exp_e{\tilde{u}_{_0}}} \mathscr{A}_\mu^s
$$
is a Fredholm operator of index zero.
\end{thm}

\noindent Note that here we assume $u_0$ is smooth (cf Remark \ref{rem:modulo}).

\begin{proof}
It is sufficient to show compactness of the algebra coadjoint operator at $\tilde{u}_0$.
Proceeding as in the proof of Lemma \ref{compactness_lemma} we have
\begin{align*}
v \to \mathrm{ad}^\ast_v \tilde{u}_0
&=
\mathrm{curl}\, \Delta^{-1} [ v, \mathrm{curl}\, \tilde{u}_0 ] + \text{finite rank}
\\
&=
\mathrm{curl}\, \Delta^{-1} [ v, \eta_\ast\mathrm{curl}\, u_0 ] + \text{finite rank}
\end{align*}
where using Lemma \ref{lem_concentration} we find
\begin{align*}
[ v, \eta_\ast u_0 ]
=
[ v, \phi\circ\eta^{-1} \eta_\ast K ]
=
d(\phi\circ\eta^{-1}) (v)
\end{align*}
since $\eta_\ast K = K$ and $[v,K] =0$.
The rest of the proof follows now as in the case of Theorem \ref{fredholmness}.
\end{proof}

In general coadjoint orbits for the three-dimensional volumorphism group are difficult to compute and understand. 
In two dimensions the basics of the theory are laid out in Arnold-Khesin~\cite{AK}. The analysis of the coadjoint 
orbits and their relationship to the space of steady solutions of the Euler equations is laid out in detail in 
Choffrut-{\v S}ver\'ak~\cite{CS}, and the complete classification of orbits in the two-dimensional situation with
arbitrary topology was completed in Izosimov-Khesin-Mousavi~\cite{IKM}. We can compute the coadjoint orbits fairly 
explicitly in simple axisymmetric cases; we will illustrate this here and leave the details and generalization for 
a future paper. 
\begin{remark} 
Consider the Heisenberg group $H\cong \mathbb{R}^3$ which fibers over the flat Euclidean plane $\mathbb{R}^2$; 
we describe it in detail later in Theorem \ref{fibrationthm}. Its geometry is determined by declaring 
the vector fields 
$e_1 = \partial_{\theta}$, 
$e_2 = \partial_x + \tfrac{1}{2} y\,\partial_{\theta}$ 
and 
$e_3 = \partial_y - \tfrac{1}{2} x \, \partial_{\theta}$ 
to be orthonormal. 
With $K=e_1$, axisymmetric volumorphisms $\Psi\colon \mathbb{R}^3\to\mathbb{R}^3$ have the form 
$$ 
\Psi(\theta,x,y) = \big( \theta+\gamma(x,y), \Phi(x,y) \big), 
\quad 
\Phi(x,y) = \big( \eta(x,y),\xi(x,y) \big), 
\quad 
\eta_x \xi_y - \eta_y\xi_x \equiv 1; 
$$ 
note that $\Phi$ is an area-preserving diffeomorphism of $\mathbb{R}^2$. 
Given an axisymmetric field $u_0\in T_e\mathscr{A}_\mu$, the coadjoint orbits are the sets 
$$ 
\mathscr{C}_{u_0} 
:= 
\big\{ u \in T_e\mathscr{A}_\mu :  \exists \Psi\in \mathscr{A}_{\mu} \text{ s.t. } \Psi^*du^{\flat} = du_0^{\flat} \big\}. 
$$ 
Writing $u_0 = g_0(x,y) e_1 - \partial_yf_0(x,y) e_2 + \partial_xf_0(x,y) e_3$ 
for some functions $f_0,g_0\colon\mathbb{R}^2\to\mathbb{R}$ which decay to zero at infinity, 
we get the simple condition for the angular component, namely 
\begin{equation} \label{gcondition} 
g\circ\Phi = g_0, 
\end{equation} 
together with the rather more complicated condition for the skew-gradient component, that is 
\begin{equation} \label{fcondition} 
\Delta f\circ \Phi + \tfrac{1}{2} \eta \{g_0, \xi\} - \tfrac{1}{2} \xi \{g_0, \eta\} + \{g_0, \gamma\} 
= 
\Delta f_0 + \tfrac{1}{2} x \partial_x g_0 + \tfrac{1}{2} y\partial_y g_0. 
\end{equation} 
The condition \eqref{gcondition} is well-understood from the usual two-dimensional coadjoint orbit analysis 
as in \cite{AK,CS,IKM}. The simplest case is when $g_0$ is a decreasing radial function with a global maximum 
at the origin, such as $g_0(r) = e^{-r^2}$. We can obtain an area-preserving $\Phi$ for a given $g$ 
satisfying \eqref{gcondition} if and only if $g$ satisfies 
\begin{equation} \label{gsolution} 
\text{Area} \big( \{(x,y) : g(x,y)>c\} \big) 
= 
\text{Area}\big(\{ (x,y) : g_0(x,y)>c\}\big) 
\quad 
\text{for all $c\in\mathbb{R}$}. 
\end{equation} 
Given such a $g$, the function $\Phi$ satisfying \eqref{gcondition} is uniquely determined up to precomposition 
with an area-preserving diffeomorphism of $\mathbb{R}^2$ that preserves all circles centered at the origin, 
i.e., $\Upsilon(r,\theta) = \big(r,\theta+\omega(r)\big)$ 
for some function $\omega(r)$. 

To understand the condition \eqref{fcondition}, note that $\gamma$ is an arbitrary function on $\mathbb{R}^2$ 
and if $g_0$ is radial then the Poisson bracket is 
$$ 
\{g_0, \gamma\} = r^{-1} g_0'(r) \frac{\partial \gamma(r,\psi)}{\partial \psi} 
$$ 
in polar coordinates $(r,\psi)$; we may solve \eqref{fcondition} as a PDE for $\gamma$ if and only if 
its integral around any circle centered at the origin is zero. 
This becomes the condition 
\begin{equation}\label{fsolution}
\int_{S_r} \Delta f\circ \Phi \, ds = \int_{S_r} \Delta f_0 \, ds \qquad \text{for all $r>0$}, 
\end{equation}
where $S_r$ is the circle of radius $r$ centered at the origin. 
It can be checked that this condition is independent of the choice of $\Phi$, i.e., of the choice of $\omega(r)$. 
Thus, for radial $g_0$ and arbitrary $f_0$, we get the necessary and sufficient conditions \eqref{gsolution} 
and \eqref{fsolution} for $f$ and $g$ to give a $u\in C_{u_0}$. 
\end{remark} 
%

\section{Examples}
\label{sec:ex}

In this section we describe examples in which the $L^2$ exponential map exhibits
Fredholm properties. In our examples the underlying manifold $M$ is modelled on one of
the eight three-dimensional geometries of Thurston each of which admits a compact realization,
cf. \cite{thurston}.
It is also convenient to assume that $M$ is equipped with a nonvanishing Killing field, however
this is not required as for example in the well known Euclidean case $\mathbb{R}^3$ with symmetry.
Along the way we derive a collection of formulas which are convenient to study the axisymmetric Euler equations
on any Riemannian 3-manifold.

\subsection{Axisymmetric vector fields on $3$-manifolds}

For simplicity we assume that the Killing field $K$ is smooth and nowhere vanishing.
%
%
Elementary vector calculus formulas for Riemannian 3-manifolds give
\begin{equation}\label{collinearcurl}
K \times \curl{K} = \grad \lvert K\rvert^2
\quad \text{and} \quad
\curl (K/|K|^2) = \phi \, K/|K|^2.
\end{equation}
where $\phi$ is a scalar-valued function. In particular, $[K,\curl{K}]=0$.

It is convenient to introduce an analogue of the skew-gradient and the Poisson bracket
for Hamiltonian functions on a symplectic manifold. They will be used throughout this section
to obtain explicit formulas involving axisymmetric vector fields.
Given $f\colon M\to\mathbb{R}$ with $K(f)=0$ we set
\begin{equation} \label{skewgrad}
\nabla^{\perp}f = \frac{K\times \grad f}{\lvert K\rvert^2}.
\end{equation}
If, in addition $K(g)=0$, then we set
\begin{equation} \label{poissonbracket}
\{f,g\}
=
\langle \nabla^{\perp}f, \nabla g\rangle
=
\frac{\langle K, \grad f\times \grad g\rangle}{\lvert K\rvert^2}.
\end{equation}
It follows that $\{g,f\} = -\{f,g\}$ and $\nabla f \times \nabla g = \{f,g\} K$.
For example, if $M=\mathbb{R}^3$ and $K = \partial_z$ on $\mathbb{R}^3$ then these definitions
reproduce the usual skew-gradient and Poisson bracket for functions.
Furthermore, we have
\begin{equation} \label{curlexpand}
\curl{K} = \phi K - \nabla^{\perp} \lvert K\rvert^2.
\end{equation}
%
%
\begin{lem} \label{divcurllemma}
	If $K(f)=K(g)=0$ then we have
	\begin{equation} \label{divdiv}
	\diver{(\nabla^{\perp}f)} = 0 = \diver{(gK)}
	\end{equation}
	\begin{equation} \label{curlcurl}
	\curl{(\nabla^{\perp} f)} = (\Delta_{K}f) K
	\quad \text{and} \quad
	\curl{(gK)} = - \nabla^{\perp}(\lvert K\rvert^2 g) + \phi g K
	\end{equation}
	where
	\begin{equation} \label{kappalaplacian}
	\Delta_K f = \diver{ ( \grad f/|K|^2 )}.
	\end{equation}
\end{lem}
\begin{proof}
A straightforward computation.
\end{proof}

The next result is a special case of a Hodge-type decomposition for axisymmetric vector fields.
For simplicity, in what follows we will assume that $H^1(M)=0$.
\begin{prop} \label{axisymmetricvectorfieldclassification}
	Any axisymmetric vector field tangent to the boundary $\partial M$ can be decomposed as
	\begin{equation} \label{udecompose}
	u=\nabla^{\perp}f + gK, \qquad \text{where} \; K(f)=K(g)=0 \;\text{and $f\vert_{\partial M}$ is constant.}
	\end{equation}
	Conversely, any $u$ of the form \eqref{udecompose} is axisymmetric and tangent to the boundary.
\end{prop}
\begin{proof}
Since $[u,K]=0$ and both fields are divergence free, we have $\mathrm{curl}\, (K\times u) = 0$,
which together with the assumption on $M$ implies that $K \times u = -\nabla f$ for some function $f$.
Clearly, $K(f)=0$ and
$ K \times (K \times u) = -\lvert K\rvert^2 \, u + \langle K, u\rangle \, K = K \times \grad f$.
Setting $g = \langle K, u\rangle/\lvert K\rvert^2$ and solving for $u$, we get \eqref{udecompose}.
	
For the converse, applying the divergence operator to \eqref{udecompose} and using Lemma \ref{divcurllemma}
we obtain $\mathrm{div}\, u =0$. Since $K$ and $\nu \times K$ are tangent to the boundary, we find that
$\langle \nabla^{\perp}f, \nu \rangle = |K|^{-2} \langle \grad f, \nu \times K \rangle = 0$
because $f|_{\partial M}$ is constant.
Finally, a direct computation gives
$[K, u] = -\curl{(K \times u)} = \curl \grad f = 0$.
\end{proof}

We know from Proposition $\ref{prop_lie_algebra}$ that if $u$ and $v$ are axisymmetric, then so is $[u,v]$.
This commutator can be written explicitly in terms of the decomposition \eqref{udecompose}.
\begin{prop} \label{commutatorexplicitprop}
Suppose $f$ and $g$ are functions on $M$ such that $K(f)=K(g)=0$. Then, we have
\begin{align*}
[\nabla^{\perp}f, \nabla^{\perp}g] = \nabla^{\perp}\{f,g\} - \{f,g\} \frac{\phi K}{\lvert K\rvert^2},
\quad
[\nabla^{\perp}f, gK] = \{f,g\} K, \label{sgradswirl}
\quad
[fK, gK] = 0.
\end{align*}
\end{prop}
\begin{proof}
Follows directly from the definitions.
\end{proof}

Using Proposition \ref{commutatorexplicitprop} we can now similarly express the coadjoint operator
$\ad_v^\ast u$ for axisymmetric vector fields $u$ and $v$.
This could be done either on $T_e\mathscr{D}_\mu^s$ or by restricting to the subalgebra $T_e\mathscr{A}_\mu^s$
of axisymmetric vector fields.
That the end result is the same is a reflection of the fact that the space of axisymmetric diffeomorphisms
is a totally geodesic submanifold, cf. Theorem \ref{prop_he_inv}.

%
%
%
%

%
\begin{prop} \label{coadjointthm}
	If $u=\nabla^{\perp}f + gK$ and $v=\nabla^{\perp}h+jK$ then
	\begin{equation} \label{coadjoint}
	\ad_v^{*}u
	=
	\{h,\sigma\} \, \frac{K}{\lvert K\rvert^2} - \nabla^{\perp}\Delta_{K}^{-1}\{j,\sigma\}
	-
	\nabla^{\perp} \Delta_{K}^{-1}\{\Delta_{K}f + \phi g, h\}
	\end{equation}
	where $\Delta_{K}^{-1}$ is defined by the zero Dirichlet boundary condition
	and $\sigma = \lvert K\rvert^2 g$.
	In particular, if $u,v\in T_e\mathscr{A}_\mu^s$ then $\ad_v^{*}u\in T_e \mathscr{A}_\mu^s$.
\end{prop}
\begin{proof}
First observe that $\Delta_{K}$ is an elliptic self-adjoint operator. If $M$ has a boundary
then $\Delta_{K}\varphi = \psi$ has a unique solution with $\varphi=0$ on the boundary;
otherwise, the solution is unique up to a constant.
In either case the operator $\nabla^{\perp}\Delta_{K}^{-1}$ is well-defined.
	
On $T_e\mathscr{D}_\mu^s$ we have $\ad_v^{*}u = -P(v\times \curl{u})$; see \cite{EMP06}.
By assumption $M$ has no harmonic fields.
Then $\curl$ is invertible on $T_e\mathscr{D}_\mu^s$ and we may write
$$
\ad_v^{*}u = \curl^{-1} \curl(-v\times \curl{u}) = \curl^{-1}[v, \curl{u}].
$$
From \eqref{curlcurl} we have $\curl{u} = (\Delta_{K}f + \phi g)\, K - \nabla^{\perp}(\lvert K\rvert^2 g)$
and applying the commutator formula of Proposition \ref{commutatorexplicitprop}, we obtain
\begin{equation*} \label{commutatoradstar}
[v,\curl{u}]
=
-\nabla^{\perp}\{h, \lvert K \rvert^2 g\}
+
\{h, \lvert K\rvert^2g\} \,\frac{K}{\lvert K\rvert}
+
\{h, \Delta_{K}f + \phi g\} K
+
\{\lvert K\rvert^2 g, j\} K.
\end{equation*}
Finally, inverting \eqref{curlcurl} on axisymmetric fields, we get
$$
\curl^{-1}(\nabla^{\perp}p + q K)
=
\nabla^{\perp} \Delta_{K}^{-1} \Big( q + \frac{\phi p}{\lvert K\rvert^2} \Big) - \frac{p}{\lvert K\rvert^2} K.
$$
Combining the above formulas gives \eqref{coadjoint}.
\end{proof}
\begin{remark} \label{rmk_compactness}
Proposition $\ref{coadjointthm}$ can be used to obtain an alternative proof of the compactness result
for the operator $v\mapsto \ad_{u_0}^{*}v$ when $u_0$ is swirl-free, c.f. Lemma \ref{compactness_lemma}.
Indeed, if $v \in T_e\mathscr{A}_\mu^s$ and $g=0$ then from \eqref{coadjoint} we get
$\ad_v^{*}u_0 = -\nabla^{\perp} \Delta_{K}^{-1}\{\Delta_{K}f, h\}$
which, since $u_0$ is smooth, is in $H^{s+1}$ thus gaining a derivative.
\end{remark}

Finally, it is useful to rewrite the Euler equations explicitly in terms of the operators $\{ \, , \, \}$ from \eqref{poissonbracket} and $\Delta_K$ from \eqref{kappalaplacian}. In the next subsection we will explore this equation in the setting of
the eight model geometries of Thurston.
\begin{thm} \label{eulerequationcor}
The incompressible Euler equations \eqref{eq:Euler} can be written as Euler-Arnold equations
on $T_e\mathscr{A}_\mu^s$ in the form
\begin{equation} \label{eulerequation}
\begin{split}
\Delta_{K} \partial_t f + \{f,\Delta_{K}f\} + \left\{f,\frac{\phi \sigma}{\lvert K\rvert^2}\right\} - \frac{1}{2} \left\{ \frac{1}{\lvert K\rvert^2}, \sigma^2\right\} &= 0
\\
\partial_t \sigma + \{f,\sigma\} &= 0
\end{split}
\end{equation}
where $u = \nabla^\perp f + \frac{\sigma K}{\lvert K\rvert^2}$ in terms of the stream function $f$ and the swirl $\sigma$.
\end{thm}
\begin{proof}
The proof follows by substituting for $u$ in the equation
$\partial_t u + \mathrm{ad}^\ast_u u =0$ and using the formula \eqref{coadjoint}
with $h=f$ and $j=g$ together with the identity
$$
\grad g\times \grad (\lvert K\rvert^2 g)
=
g\grad g\times \grad \lvert K\rvert^2
=
g\grad g\times (K\times \curl{K}) = g\curl{K}(g) K.
$$
By formulas \eqref{poissonbracket}--\eqref{curlexpand}, together with the fact that $K(g)=0$,
we obtain
\begin{align*}
g \curl{K}(g) &= g\phi K(g) - g\langle \nabla^{\perp}\lvert K\rvert^2, \nabla g\rangle = -g\{\lvert K\rvert^2, g\} \\
&= -\frac{\sigma}{\lvert K\rvert^2} \, \left\{ \lvert K\rvert^2, \frac{\sigma}{\lvert K\rvert^2}\right\} = \frac{1}{2} \left\{ \frac{1}{\lvert K\rvert^2}, \sigma^2\right\}.
\end{align*}
\end{proof}
%

\subsection{The Euler equations in Thurston geometries}
\label{sec_examples}

Each choice of a $3$-manifold $M$ (compact or not) and a Killing vector field $K$
leads to a system of Euler-Arnold equations \eqref{eulerequation}.
If $K$ is allowed to vanish then the system \eqref{eulerequation} may be singular (its coefficients suitably reinterpreted)
on the zero set $\mathcal{Z}$ of $K$, just as in the case of the standard axisymmetric Euler equations in $\mathbb{R}^3$. 
To describe the two simplest cases, recall that $\phi$ is defined as in the previous section, cf (\ref{collinearcurl}). Then we have:
\begin{itemize}
\item $\phi\equiv 0$, or in other words $\langle K, \curl{K}\rangle = 0$. We will see that this is the integrable case when the swirl-free
diffeomorphisms form a subgroup. In this case the motion is driven by the gradient of $\lvert K\rvert^2$.
\item Both $\phi$ and $\lvert K\rvert$ are constant. In this case we get constant-coefficient equations, and the geometry involves a fibration of a contact manifold over a symplectic surface.
\end{itemize}

\subsubsection{The integrable case: rotations and translations}

We first analyze the case $\phi\equiv 0$.
It is worth recalling from Section \ref{sec:swirlfree} that depending on the properties of
the normal distribution $K^\perp$ in \eqref{eq:K-int} two types of extreme cases that can appear.
If this distribution is integrable on $M \setminus \mathcal{Z}$ then we have a submanifold of swirl-free diffeomorphisms.

By Proposition \ref{axisymmetricvectorfieldclassification}, axisymmetric divergence-free vector fields are expressed
in terms of functions $f$ and $g$ satisfying $K(f)\equiv K(g)\equiv 0$, and so $f$ and $g$ must be constant
along orbits of $K$. This imposes very different constraints depending on whether orbits of $K$ are dense
in $M$ or not. In what follows, we will assume that $K$ is \emph{regular}, as defined below.
\vskip 0.5cm
\hspace{-0.2cm}\textbf{(R)} For each $x \in M$ with $K(x) \neq 0$, there is a neighborhood $U$ of $x$ such that the intersection $L_x \cap U$ is homeomorphic to $\mathbb{R}$, where $L_x$ is the maximal integral curve of $K$ through $x$.
\vskip 0.5cm
This condition ensures that orbits of $K$ cannot be dense in any set, so that they must be either closed circles in the compact case or lines escaping compact sets in the noncompact case. If $K$ is not regular, there may be very few axisymmetric diffeomorphisms.

In light of the formulas derived in Section $\ref{sec:ex}$, we revisit the distribution, previously introduced in Section $\ref{sec:swirlfree}$, given by
\begin{equation}\label{eq_appendix_K}
M {\setminus} \mathcal{Z} \ni x \to K^\perp_x = \big\{ V \in T_xM : \langle V, K(x) \rangle = 0 \big\}
\end{equation}
on all points $M {\setminus} \mathcal{Z}$ outside its zero set $\mathcal{Z}$.

\begin{theorem}\label{integrable}
	Suppose that the function $\phi$ defined by \eqref{collinearcurl} is identically zero on $M{\setminus}\mathcal{Z}$. Then, the distribution $K^{\perp}$	is integrable, and there is a family of submanifolds $N_{\theta}$ of $M{\setminus}\mathcal{Z}$ on which $K^{\flat}$ vanishes. Functions $f$ satisfying $K(f)\equiv 0$ may be specified arbitrarily on any individual $N_{\theta}$ and then extended to $M{\setminus}\mathcal{Z}$ using the flow of $K$.
\end{theorem}

The integrability claim in Theorem \ref{integrable} follows from the first formula in Proposition \ref{commutatorexplicitprop} and the Frobenius theorem. The second claim is clear from our assumption that the flow of $K$ is transitive on leaves (see Section $4$, the discussion after Proposition $\ref{prop_flow_K}$).

The doubly-warped metrics mentioned in Definition \ref{rotationlike} are rich enough to include the simplest Killing fields in
all Thurston geometries for which $\phi=0$. We will write the equations for both rotations and translations in the same form; thus
in the coordinate chart $(r,\theta,z)$ below, we will always have $K = \partial_{\theta}$, but the domain of $\theta$ will sometimes be $S^1$
and sometimes $\mathbb{R}$. We will assume $M$ is simply connected for simplicity, but almost all of this could be done on quotients with more interesting topology.

\begin{theorem}\label{rotationtranslation}
If $M$ has a coordinate chart $(r,\psi,\theta)$ with Riemannian metric given by the doubly-warped product
$$ ds^2 = dr^2 + \alpha(r)^2 \, d\psi^2 + \beta(r)^2 \, d\theta^2,$$
then $K = \partial_{\theta}$ is a Killing field. Suppose either $\theta \in S^1$ and $\alpha$ and $\beta$
satisfy the rotation conditions as in Definition \ref{rotationlike}, or that $\theta\in \mathbb{R}$ with $\alpha$ nowhere zero,
and suppose that the manifold is oriented so that $dV = \alpha(r) \beta(r) \, dr\wedge d\psi \wedge d\theta$ is the volume form.

Then functions $f$ on $M$ satisfying $K(f)=0$ are given by functions of $r$ and $\psi$, arbitrary except for the condition that
$f$ extends to an even function through $r=0$ in the rotational case.
The Poisson bracket \eqref{poissonbracket} and $K$-Laplacian \eqref{kappalaplacian} are given by
\begin{align*}
\{f,g\} &= \frac{1}{\alpha(r)\beta(r)} \, \left( \frac{\partial f}{\partial r} \, \frac{\partial g}{\partial \psi}  - \frac{\partial f}{\partial \psi} \, \frac{\partial g}{\partial r}\right) \\
\Delta_Kf &= \frac{1}{\alpha(r)\beta(r)} \, \frac{\partial}{\partial r} \Big( \frac{\alpha(r)}{\beta(r)} \, \frac{\partial f}{\partial r}\Big)
+ \frac{1}{\alpha(r)^2 \beta(r)^2} \, \frac{\partial^2 f}{\partial \psi^2},
\end{align*}
while the axisymmetric Euler equation \eqref{eulerequation} is given by 
\begin{align*}
\partial_t \Delta_Kf + \{f, \Delta_Kf\} + \frac{\beta'(r)}{\alpha(r)\beta(r)^4} \, \frac{\partial}{\partial \psi}(\sigma^2), 
\quad 
\partial_t\sigma + \{f,\sigma\} = 0.
\end{align*}
\end{theorem}

\begin{proof}
In this metric an oriented orthonormal basis is given by $e_1 = \partial_r$, $e_2 = \frac{1}{\alpha(r)} \, \partial_{\psi}$, and $e_3 = \frac{1}{\beta(r)}\, \partial_{\theta}$. We have $K^{\flat} = \beta(r)^2 \, d\theta$, so that $d(\frac{K^{\flat}}{\lvert K\rvert^2}) = d^2\theta = 0$, which implies that $\phi = 0$ in equation \eqref{collinearcurl}. Functions $f$ with $K(f)\equiv 0$ are determined by their values on any set $\theta = \text{constant}$. We have $\grad f = f_r \, e_1 + \frac{1}{\alpha(r)} f_{\psi} \, e_2$, so that the Poisson bracket
\eqref{poissonbracket}
\begin{gather}
\begin{split}
\{f,g\} &= \frac{\langle K, \grad f\times \grad g\rangle}{\lvert K\rvert^2} \\
&= 
\frac{1}{\beta(r)} \big\langle e_3, \big( f_r \, e_1 + \tfrac{1}{\alpha(r)} f_{\psi} \, e_2 \big) 
\times 
\big( g_r \, e_1 + \tfrac{1}{\alpha(r)} g_{\psi} \, e_2 \big) \big\rangle \\
&= 
\frac{f_r g_{\psi} - f_{\psi}g_r}{\alpha(r)\beta(r)}.
\end{split}
\end{gather} 
The formula for the $K$-Laplacian follows from the formula
$$ \diver{\left(\frac{\nabla f}{\lvert K\rvert^2}\right)} =
\frac{1}{\mu(r)} \, \frac{\partial}{\partial r}\left( \frac{\mu(r)}{\beta(r)^2}\,\frac{\partial f}{\partial r}\right)
+ \frac{1}{\mu(r)} \, \frac{\partial}{\partial \psi} \left( \frac{\mu(r)}{\beta(r)^2} \, \frac{1}{\alpha(r)^2} \, \frac{\partial f}{\partial \psi}\right),$$
since here the Riemannian volume form is $dV = \mu(r) \, dr\wedge d\psi \wedge d\theta$ for $\mu(r) = \alpha(r)\beta(r)$.
The Euler equation \eqref{eulerequation} is now straightforward.
\end{proof}

Now we list the examples of the Thurston geometries for which this applies. Note that in some of these cases the standard
orientation is the opposite of the one in Theorem \ref{rotationtranslation}, which can be fixed by replacing $K$ with $-K$. The
only effect of this is to change the sign of the stream function $f$ and of the Poisson bracket \eqref{poissonbracket}, so that
the sign of the $\frac{\partial}{\partial \psi}$ term changes.
\begin{itemize}
\item Euclidean $\mathbb{R}^3$. The translation has $(r,\psi,\theta)\in \mathbb{R}^3$ with $\alpha(r)=\beta(r)=1$, while the rotation has $(r,\psi,\theta) \in \mathbb{R}_+\times \mathbb{R}\times S^1$ with $\alpha(r)=1$ and $\beta(r)=r$. In the rotation case, the usual orientation has the opposite sign. Functions with $K(f)=0$ are defined on $\mathbb{R}^2$ or the half-plane $\mathbb{R}_+\times \mathbb{R}$ respectively.
\item Unit $3$-sphere $\mathbb{S}^3$. Here we have rotation with $\theta\in S^1$, for $\alpha(r)=\cos{r}$ and $\beta(r)=\sin{r}$. Here $\psi\in S^1$ and $r\in (0,\frac{\pi}{2})$, so the submanifold where $K$-invariant functions are defined is equivalent to the upper hemisphere of radius $1$.
\item Hyperbolic space $\mathbb{H}^3$. The standard isometries are elliptic, hyperbolic, or parabolic. The choice $\alpha(r)=\cosh{r}$ and $\beta(r)=\sinh{r}$ with $(r,\psi,\theta)\in \mathbb{R}_+\times \mathbb{R}\times S^1$ gives a rotation. Using $\alpha(r)=\sinh{r}$ and $\beta(r)=\cosh{r}$ with $(r,\psi,\theta)\in \mathbb{R}_+\times S^1 \times \mathbb{R}$ gives a translation which is hyperbolic (corresponding to rescaling symmetry in the upper half-space model). Finally using $\alpha(r)=\beta(r) = e^{-2r}$ for $(r,\psi,\theta)\in \mathbb{R}^3$ gives a parabolic translation (corresponding to translation symmetry in the upper half-space model). In all three cases, the invariant functions are defined on the hyperbolic plane $\mathbb{H}^2$.
\item The product $\mathbb{S}^2\times \mathbb{R}$. Here $\alpha(r)=\sin{r}$ and $\beta(r)=1$ gives the obvious translation in the $\mathbb{R}$ direction with a $2$-sphere as the submanifold domain, while $\alpha(r)=1$ and $\beta(r)=\sin{r}$ gives the usual rotation of the $2$-sphere, with $(r,\psi,\theta)\in \mathbb{R}_+\times \mathbb{R}\times S^1$, so that invariant functions are defined on a flat half-plane.
\item The product $\mathbb{H}^2\times \mathbb{R}$. As on $\mathbb{H}^3$ we have three simple choices: $\alpha(r)=\sinh{r}$ and $\beta(r)=1$ gives the obvious translation. The choice $\alpha(r)=1$ and $\beta(r)=e^{-2r}$ gives the hyperbolic translation isometry, where the submanifold is flat $\mathbb{R}^2$, while the choice $\alpha(r)=e^{-2r}$ and $\beta(r)=1$ gives the parabolic translation isometry with submanifold $\mathbb{H}^2$.
\item The solvable group $\mathrm{Sol}$. Here the metric has $\alpha(r)=e^{-2r}$ and $\beta(r)=e^{2r}$ with $(r,\psi,\theta)\in\mathbb{R}^3$. The domain submanifold is $\mathbb{H}^2$. Replacing $r$ with $-r$ obviously has no significant effect; we get translations either way.
\end{itemize}

Note that the only real difference between rotations and translations is that rotations involve a more singular term in the denominator; even the difference $\theta\in S^1$ or $\theta\in\mathbb{R}$ is illusory since we can consider periodic translations. The metric on the surface is always $ds^2 = dr^2 + \alpha(r)^2 \, d\psi^2$, but since the induced area form is $dV = \alpha(r) \beta(r) \, dr\wedge d\psi$ (all volume integrals for $\theta$-independent functions will reduce to area integrals in the $(r,\psi)$ variables), the fluid effectively acts as though it has a ``mass density'' of $\beta(r)$.

\subsubsection{The nonintegrable case: fibrations over surfaces}

In the previous section the equations were essentially similar to the standard Euclidean axisymmetric situation. However the non-Euclidean situation allows also for an essentially different situation where $\phi=-1$ and $\lvert K\rvert = 1$; here the distribution $K^{\perp}$ is nowhere integrable, and the Euler equation \eqref{eulerequation} is ``driven'' by the twisting of a contact structure rather than the gradient of $\lvert K\rvert^2$. The most famous example is the Hopf fibration of $\mathbb{S}^3$ over $\mathbb{S}^2$, but there are similar fibrations of $SL_2(\mathbb{R})$ over $\mathbb{H}^2$ and of the Heisenberg group $\mathrm{Nil}$ over $\mathbb{R}^2$. We will present all three in a unified way.

\begin{theorem}\label{fibrationthm}
Let $k \in \{1,0,-1\}$, and let $M$ be a simply-connected Lie group with left-invariant Riemannian metric defined by setting left-invariant vector fields $\{e_1,e_2,e_3\}$ to be orthonormal, where the Lie brackets satisfy
\begin{equation}\label{brackets}
[e_1,e_2] = -ke_3, \qquad [e_3,e_1] = -ke_2, \qquad [e_2,e_3] = -e_1.
\end{equation}
Then $K=e_1$ is a Killing field of unit length, and there is a smooth projection map $P\colon M\to N_k$ where $N_k$ is a simply-connected model surface of constant curvature $k$: that is, $N_1 = \mathbb{S}^2$, $N_0 = \mathbb{R}^2$, and $N_{-1} = \mathbb{H}^2$. Functions $f\colon M\to\mathbb{R}$ with $K(f)\equiv 0$ descend to functions on $N_k$.
The Poisson bracket \eqref{poissonbracket} and $K$-Laplacian on $N_k$ are given by
the standard ones on $N_k$, and the Euler equation \eqref{eulerequation} takes the form
\begin{equation}\label{eulerliegroup}
\partial_t \Delta f + \{f, \Delta f\} + \{f,\sigma\} = 0, \qquad \partial_t \sigma + \{f,\sigma\} = 0.
\end{equation}
\end{theorem}

\begin{proof}
To verify $K$ is a Killing field, we check the Killing equation $\langle \nabla_uK,v\rangle + \langle \nabla_vK,u\rangle$; since $K=e_1$ and the brackets are cyclic, the only nontrivial one is $u=e_2$ and $v=e_3$, for which
\begin{gather}
\begin{split}
\langle \nabla_{e_2}e_1,e_3\rangle + \langle \nabla_{e_3}e_1,e_2\rangle &= -\langle [e_1,e_2], e_3\rangle + \langle [e_3,e_1], e_2\rangle \\
&= k-k=0.
\end{split}
\end{gather}
We compute
$$ \phi = \langle \curl{e_1}, e_1\rangle = de_1^{\flat}(e_2,e_3) = -\langle e_1, [e_2,e_3]\rangle = 1,$$
so that $e_1$ is a curl eigenfield with eigenvalue $1$ regardless of $k$.

Simply-connected Lie groups are completely determined by their Lie brackets, and we have simple well-known representations of those groups in each case. When $k=1$ the Lie group is the special unitary group $SU(2)$, described by complex matrices
$$ SU(2) = \Big\{ \big(\begin{smallmatrix}
w & -\overline{z} \\
z & \overline{w}\end{smallmatrix}\big) \, \Big| \, \lvert w\rvert^2 + \lvert z\rvert^2 = 1\Big\}$$
with Lie algebra given by
$$ e_1 = \tfrac{1}{2} \big(\begin{smallmatrix} i & 0 \\ 0 & -i\end{smallmatrix}\big), \qquad e_2 = \tfrac{1}{2} \big( \begin{smallmatrix} 0 & -1 \\
1 & 0 \end{smallmatrix}\big), \qquad e_3 = \tfrac{1}{2} \big(\begin{smallmatrix} 0 & i \\ i & 0 \end{smallmatrix}\big),$$
which satisfies \eqref{brackets}. The group $SU(2)$ is diffeomorphic to $\mathbb{S}^3$, and the Hopf fibration is $P\colon (w,z)\mapsto (2w\overline{z}, \lvert w\rvert^2 - \lvert z\rvert^2)\in \mathbb{C}\times \mathbb{R}$, which maps surjectively to $\mathbb{S}^2$, with a circle as the preimage of each point. Each circle is the orbit of $K$. The left-invariant metric is defined by $ds^2 = 4(\lvert dw\rvert^2 + \lvert dz\rvert^2)$.

When $k=-1$ the Lie group is the special linear group $SL_2(\mathbb{R})$, described by real matrices with unit determinant
$$ SL_2(\mathbb{R}) = \Big\{ \big(\begin{smallmatrix}
a & b \\
c & d\end{smallmatrix}\big) \, \Big| \, ad-bc = 1\Big\},$$
with Lie algebra given by
$$
e_1 = \tfrac{1}{2}  \big(\begin{smallmatrix} 0 & 1 \\ -1 & 0\end{smallmatrix}\big),
 \qquad
 e_2 = \tfrac{1}{2} \big( \begin{smallmatrix} 1 & 0 \\
0 & -1 \end{smallmatrix}\big),
\qquad
e_3 = \tfrac{1}{2} \big(\begin{smallmatrix} -1 & 0 \\ 0 & -1 \end{smallmatrix}\big).
$$
The brackets are easily seen to satisfy \eqref{brackets} with $k=-1$. The left-invariant metric is given by
$$
ds^2 = 2\Tr{\big((A^{-1}dA)^T (A^{-1}dA)\big)},
$$
and the isometries in the $K$ direction are given by the orthogonal matrices
$
\Big( \begin{smallmatrix} \cos{\theta} & -\sin{\theta}
\\
\sin{\theta} & \cos{\theta}\end{smallmatrix} \Big)
$;
the projection map is $A\mapsto A A^T$, which maps to the positive-definite symmetric matrices surjectively, with circle fibers, and the trace metric is a Riemannian submersion onto the hyperbolic plane.

Finally when $k=0$ the Lie group is the Heisenberg group of upper triangular $3\times 3$ matrices with $1$ on the diagonal
$$
H = \Big\{ \Big( \begin{smallmatrix}
1 & x & z \\
0 & 1 & y \\
0 & 0 & 1
\end{smallmatrix} \Big) \, \Big| \, x,y,z\in\mathbb{R} \Big\}.
$$
The Lie algebra can be generated by
$$
e_1 = \Big( \begin{smallmatrix} 0 & 0 & 1 \\ 0 & 0 & 0 \\ 0 & 0 & 0 \end{smallmatrix} \Big),
\qquad
e_2 = \Big( \begin{smallmatrix}
0 & 1 & 0 \\
0 & 0 & 0 \\
0 & 0 & 0 \end{smallmatrix} \Big),
\qquad
e_3 = \Big( \begin{smallmatrix} 0 & 0 & 0 \\ 0 & 0 & -1 \\ 0 & 0 & 0
\end{smallmatrix} \Big),
$$
and again we check the bracket relations \eqref{brackets} with $k=0$. The projection to $\mathbb{R}^2$ comes from forgetting about $z$; clearly the preimage of any point in $\mathbb{R}^2$ is a line.
\end{proof}

Our choices for the Lie algebra are slightly nonstandard because we want to treat all three groups the same way, and because we want to obtain spaces of constant curvature $k\in \{1,0,-1\}$ as the quotient. Roughly speaking, both the Hopf fibration and the matrix map $A\mapsto AA^T$ involve squaring terms, so that vectors will be stretched by a factor of $2$ at the identity; hence we need factors of $\tfrac{1}{2}$ everywhere to cancel those out.

These formulations are somewhat easier to see in coordinates.
For this purpose, we recall the definition of generalized trigonometric functions:
\begin{equation*} \label{generalizedtrig}
 s_k(r) = \begin{cases} \sin{r} & k=1, \\
r & k=0, \\
\sinh{r} & k=-1, \end{cases}
\quad
c_k(r) = \begin{cases} \cos{r} & k=1, \\
1 & k=0, \\
\cosh{r} & k=-1, \end{cases}
\quad
t_k(r) = s_k(r)/c_k(r).
\end{equation*}
These satisfy obvious identities such as $s_k'(r) = c_k(r)$ and $c_k(r)^2 + ks_k(r)^2=1$.
\begin{cor}\label{liegroupcoords}
We may choose coordinates $(r,\psi,\theta)$ on the three Lie groups in Theorem \ref{fibrationthm} such that
\begin{equation} \label{liegroupmetric}
ds^2 = dr^2 + s_k(r)^2 \, d\psi^2 + \big( d\theta + 2 s_k(\tfrac{r}{2})^2 \, d\psi\big)^2.
\end{equation}
The map $(r,\psi,\theta)\mapsto (r,\psi)$ is a Riemannian submersion onto the model space $N_k$
with metric $ds^2 = dr^2 + s_k(r)^2 \, d\psi^2$ on the quotient surface.
The left-invariant vector fields satisfying \eqref{brackets} are given explicitly by
\begin{equation} \label{vectorfieldcoords}
\begin{split}
e_1
&=
\partial_{\theta}
\\
e_2
&=
\cos{(\psi + k \theta)} \, \partial_r
-
\frac{\sin{(\psi+k\theta)}}{s_k(r)} \, \partial_{\psi} +
t_k(\tfrac{r}{2}) \sin{(\psi + k\theta)}\, \partial_{\theta}
\\
e_3 &= \sin{(\psi + k \theta)} \, \partial_r +
\frac{\cos{(\psi+k\theta)}}{s_k(r)} \, \partial_{\psi}
-
t_k(\tfrac{r}{2}) \cos{(\psi + k\theta)}\, \partial_{\theta}.
\end{split}
\end{equation}
In these coordinates the Poisson bracket \eqref{poissonbracket} and $K$-Laplacian \eqref{kappalaplacian}
are given by
\begin{align*}
\{f,g\} &= e_2(f)e_3(g) - e_3(f)e_2(g) = \frac{f_r g_{\psi} - f_{\psi}g_r}{s_k(r)}
\\
\Delta f &= e_2(e_2(f)) + e_3(e_3(f)))
=
\frac{1}{s_k(r)} \, \frac{\partial}{\partial r}\Big( s_k(r) \, \frac{\partial f}{\partial r}\Big)
+
\frac{1}{s_k(r)^2} \, \frac{\partial^2 f}{\partial \psi^2}.
\end{align*}
\end{cor}
\begin{proof}
The explicit coordinate charts here are given as follows.
For $SU(2)$ we have $w = \cos{\tfrac{r}{2}} e^{i\theta/2}$ and $z = \sin{\tfrac{r}{2}} e^{i\psi} e^{i\theta/2}$.
For $SL_2(\mathbb{R})$ we have
$$
A
=
R_{\psi/2} \bigg( \begin{matrix} e^{r/2} & 0 \\ 0 & e^{-r/2} \end{matrix} \bigg) R_{-\psi/2} R_{\theta/2},
$$
where $R_{\theta} = \big( \begin{smallmatrix} \cos{\theta} & -\sin{\theta} \\ \sin{\theta} & \cos{\theta}\end{smallmatrix}\big)$.
Finally for the Heisenberg group we have
$$
A = \left( \begin{matrix} 1 & r\cos{\psi} & -\theta + \tfrac{1}{2} r^2 \sin{\psi}\cos{\psi} \\
0 & 1 & r\sin{\psi} \\
0 & 0 & 1 \end{matrix}\right).
$$
We can verify directly in the coordinates that the Lie brackets of the fields \eqref{vectorfieldcoords}
satisfy the relations \eqref{brackets}, and that the fields are orthonormal in the metric \eqref{liegroupmetric},
using the formula $s_k(r) = 2s_k(\tfrac{r}{2}) c_k(\tfrac{r}{2})$.

Since the projection in each case involves forgetting about $\theta$, vertical vectors are proportional to $\partial_{\theta}$;
horizontal fields in this metric will be spanned by $\partial_r$ and $\partial_{\psi} - 2s_k(\tfrac{r}{2})^2 \,\partial_{\theta}$
and since the metric of a Riemannian submersion is determined by the horizontal vectors, we obtain the quotient metric.
The formulas for the Poisson bracket and Laplacian are standard for Lie groups with orthonormal left-invariant bases
and metrics in constant-curvature spaces with polar coordinates.
\end{proof}

In the language of Riemannian contact geometry (see Blair~\cite{blair}), the contact form is $\alpha = e_1^{\flat} = d\theta + 2s_k^2(\tfrac{r}{2}) \, d\psi$ since
$$\alpha \wedge d\alpha = s_k(r) \, dr \wedge d\psi \wedge d\theta$$
is the volume form (and in particular nowhere zero). This is a $K$-contact form and $K$ is the Reeb field, and the Riemannian metric is associated to the contact form, while the projection $P\colon M\to N_k$ is a Boothby-Wang fibration.

\end{document}